\documentclass{article}
\usepackage{amsmath,amssymb,amsthm}
\usepackage{xcolor}
\usepackage{hyperref}
\hypersetup{colorlinks,breaklinks,
             linkcolor=blue,urlcolor=blue,
             anchorcolor=blue,citecolor=blue}
\usepackage[nameinlink, capitalise, noabbrev]{cleveref}
\usepackage{comment}
\usepackage{bm}
\usepackage{mathtools}
\usepackage{enumitem}
\usepackage{graphicx}
\usepackage{caption}
\usepackage{subcaption}
\usepackage{overpic}
\usepackage{graphicx}

\newtheorem{theorem}{Theorem}[section]
\newtheorem{definition}[theorem]{Definition}
\newtheorem{lemma}[theorem]{Lemma}
\newtheorem{remark}[theorem]{Remark}
\newtheorem{proposition}[theorem]{Proposition}
\newtheorem{corollary}[theorem]{Corollary}
\newcommand{\calF}{\mathcal{F}}
\newcommand{\Tr}{\mathrm{Tr}}
\usepackage{mathrsfs}  
\newcommand{\TP}{|\!\!\>|\!\!\>|}

\newcommand{\an}[1]{{\color{orange}{AN: #1}}}

\title{Stochastic Modified Equations for Stochastic Gradient Descent in Infinite-Dimensional Hilbert Spaces}

\author{Sandra Cerrai, Qin Li, Anjali Nair, Jaeyoung Yoon}
\date{\today}

\begin{document}

\maketitle

\begin{abstract}
Inverse problems in scientific computing often require optimization over infinite-dimensional Hilbert spaces. A commonly used solver in such settings is stochastic gradient descent (SGD), where gradients are approximated using randomly sampled sub-objective functions. In this work we study the continuous-time limit of SGD in the small step-size regime. We show that the discrete dynamics can be approximated by a stochastic differential equation (SDE) driven by cylindrical Brownian motion. The analysis extends diffusion-approximation results previously established in Euclidean spaces~\cite{li2019stochastic} to the infinite-dimensional setting.

Two analytical difficulties arise in this extension. First, the cylindrical nature of the noise requires establishing well-posedness of the resulting stochastic evolution equation through appropriate structural conditions on the covariance operator. Second, since the randomness in SGD originates from discrete sampling while the limiting equation is driven by Gaussian noise, the comparison between the two dynamics must be carried out in a weak sense. We therefore introduce a suitable class of smooth functionals on the Hilbert space and prove that the discrepancy between SGD and the limiting SDE, when evaluated through these functionals, is of second order in the step size. Numerical experiments confirm the predicted convergence behavior.
\end{abstract}

\section{Introduction}


Optimization in infinite-dimensional spaces presents challenges beyond those encountered in finite dimensions. Many numerical solvers possess a computational complexity that depends explicitly on the ambient dimension $d$, rendering their direct extension to the case $d = \infty$ ill-defined.

Yet, many optimization problems are naturally posed over Hilbert spaces, so the dimension has to be infinity. A big source of such applications are inverse problems \cite{tarantola2005,engl1996,vogel2002computational,kaipio2005statistical,stuart2010inverse}. Inverse problems are ubiquitous in scientific computing, where measurements are collected to infer unknown parameters by minimizing the mismatch between simulated data and observations. Examples include classical X-ray tomography, optical tomography, and inverse scattering, where the unknowns are functions rather than finite vectors \cite{natterer2001mathematics,arridge1999optical,colton1998inverse,arridge2019solving}. In all these settings, the parameter belongs naturally to a function space, and the associated optimization problem is posed over an infinite-dimensional Hilbert space $H$. Extending optimization solvers from Euclidean space to $H$ requires a rigorous examination of operator compactness, regularity, and well-posedness.

In this work, we investigate one such extension: Stochastic Gradient Descent (SGD) in a Hilbert space. In finite dimensions, SGD is widely used when the objective function is composed of many (possibly infinitely many) sub-objective functions \cite{robbins1951stochastic,Bottou2018,lecun2002efficient}. At each iteration, the algorithm updates the iterate using a small step in the direction of the gradient of a randomly selected sub-objective, providing an unbiased estimator of the full gradient. Its low per-iteration cost, ease of implementation, and favorable stability properties have made it a very popular tool in large-scale optimization \cite{sutskever2013importance,hardt2016train,zhang2016understanding,keskar2016large,mei2018mean,sirignano2022mean}.

A powerful perspective for understanding SGD's performance is the continuous-in-time approximation: under suitable scaling of the step size $\eta$, the discrete SGD iteration can be approximated by a stochastic differential equation (SDE), often termed a Stochastic Modified Equation (SME) \cite{li2017,li2019stochastic}. This viewpoint provides insight into understanding SGD performance, including the fluctuation behavior, convergence in iterations, and long-time statistical properties, and has led to refined algorithmic modifications in the finite-dimensional setting~\cite{mandt2017stochastic,li2017,li2019stochastic,li2021validity}.



The goal of this paper is to extend this framework to infinite-dimensional Hilbert spaces, where the limiting object is an infinite-dimensional stochastic evolution equation (sometimes referred to as an SPDE) driven by cylindrical noise. One immediate difficulty lies in the lost of well-posedness of the limiting equation. In the infinite-dimensional setting, the covariance structure of the stochastic gradient must be represented by an operator on the Hilbert space, and the noise appears as a cylindrical Brownian motion. Controlling the variance of the resulting SPDE therefore requires careful control of the covariance operator to ensure the limiting SPDE makes sense. Furthermore, the two dynamics are equipped with different random sources: the randomness in SGD arises from discrete sampling while the limiting SPDE is driven by Gaussian noise, the approximation cannot be understood pathwise. Instead, the comparison must be carried out in a weak sense, through evaluation against suitable smooth functionals. This necessitates a precise identification of the appropriate class of test functionals on the Hilbert space.

We nevertheless provide an affirmative answer in a weak sense. More precisely, we define
\begin{itemize}
    \item $\phi_n$: the $n$-th step SGD iterates,
    \item $\varphi(t)$: the solution to the associated SPDE at time $t$, and 
    \item $\widetilde\varphi_n$: the Euler-Maruyama (EM) discretization evaluated at $t_n = n\eta$ where $\eta$ is a small time-stepping.
\end{itemize}
We measure the difference between these quantities against smooth test functionals $g: H \to \mathbb{R}$.




We find $\mathbb{E}(g(\phi_n)) -\mathbb{E}g(\widetilde\varphi_n)$ is of second order in $\eta$. This is a somewhat surprising conclusion. Indeed it is well-established that the EM discretization is only a first-order weak approximation of the continuous SPDE, namely $\widetilde{\varphi}_n\approx \varphi(t_n)$ only up to $O(\eta)$. Yet, the actual SGD algorithm $\phi_n$ aligns with the EM discretization to a higher order ($O(\eta^2)$)! This higher-order agreement is noteworthy because it suggests that ``vanilla" SGD possesses a high-fidelity affinity to its numerical proxy without the need for complex high-order schemes.




Two ingredients are crucial in this analysis. First, the covariance operator associated with the stochastic gradient plays a central role. Under natural structural assumptions on the sub-objective functions, this operator is compact, which allows the cylindrical noise to be interpreted within the classical theory of infinite-dimensional stochastic evolution equations. Second, the weak-error analysis requires the propagation of smoothness along both the discrete and continuous dynamics. Even if the test functional $g$ is smooth initially, its composition with the flow can deteriorate regularity; therefore, careful control of higher-order derivatives is required to close the estimates and ensure the $O(\eta^2)$ bound holds.

We summarize our main contributions of this paper. We rigorously derive the SPDE limit for SGD in $H$ and establish the well-posedness of the resulting evolution equation under cylindrical white noise. We prove that $\phi_n$ and $\widetilde{\varphi}_n$ coincide up to $O(\eta^2)$ in the weak sense, provided certain smoothness of the test functional $g$, justifying the SME as a high-fidelity proxy for algorithmic behavior. Finally, we provide experiments on inverse problems confirming the $O(\eta^2)$ error slope in infinite-dimensional settings.

The remainder of the paper is organized as follows. In \Cref{sec:preliminary} we review the functional analytic and probabilistic preliminaries. In \Cref{Sec_SME} we derive the SPDE limit and establish its well-posedness. \Cref{Sec_anal} is devoted to the weak error analysis comparing SGD and the SPDE. \Cref{sec:numerics} presents numerical evidence illustrating the theoretical results.

\section{Preliminaries}\label{sec:preliminary}

This section collects the technical ingredients required for the analysis. Two main components are needed: the probabilistic framework for stochastic evolution equations in Hilbert spaces, and the theory of high order differentiation for functionals defined on infinite-dimensional spaces. These tools are essential for formulating the limiting SPDE and for carrying out the weak error analysis between the discrete SME dynamics and SGD solver.

For completeness and to fix notation, we briefly review the relevant concepts and results that will be used throughout the paper.

\subsection{Stochastic Framework in a Hilbert Space}
We summarize the notions needed for establishing probabilistic framework in a Hilbert space. Definitions are mostly follow~\cite{Prato_Zabczyk_2014,PR2007,LPS2014,gawarecki2010}.

Over a probability space  $(\Omega,\mathcal F,\mathbb P)$, we first define a Hilbert space
\begin{equation}\label{eqn:L_omega_space}
L_\Omega^2:=\{u:\Omega\to\mathbb R~\big|~\mathbb E[u(\omega)^2]<\infty\}\,,
\end{equation}
equipped with the inner product $\langle u_1,u_2\rangle_{L_\Omega^2}=\mathbb E[u_1(\omega)u_2(\omega)]$.

This definition can be extended to higher-dimensional space. Denote $H$ a real separable Hilbert space, endowed with the inner product $\langle \cdot,\cdot\rangle_H$ and the associated norm $\|\cdot\|_H$,  and let $\mathcal B(H)$ be its Borel $\sigma$-algebra. When there is no risk for ambiguity, we drop the subscript $H$. Then over the same probability space  $(\Omega,\mathcal F,\mathbb P)$, we collect all $H$-valued $\mathcal F$-measurable random variables $X:\Omega\to H$, and for every $1\le p<\infty$, define
\begin{align*}
    \|X\|_{L^p(\Omega,H)}:=\left(\mathbb E\,\|X\|^p\right)^{1/p}=\left(\int_{\Omega}\|X(\omega)\|^pd\mathbb P(\omega)\right)^{1/p}\,,
\end{align*}
and define a Banach space:
\begin{align*}
    L^p(\Omega,H)=\{X:\Omega\to H\,,\quad \mathbb E\,\|X\|^p<\infty\}\,.
\end{align*}

We can also define second-order structures of an $H$-valued random variable $X$.
\begin{definition}
    A linear operator $Q:H\to H$ is called the {\em covariance operator} of an $H$-valued random variable $X$ if
    \begin{align*}
        \langle Q\phi,\psi\rangle = \mbox{{\em Cov}}\,(\langle X,\phi\rangle,\langle X,\psi\rangle),\quad~\phi,\psi\in H.
    \end{align*}
\end{definition}
We note that $\langle X,\phi\rangle$ is a scalar, so $\mbox{{\em Cov}}\,(\langle X,\phi\rangle,\langle X,\psi\rangle)$ is understood in the usual sense.
\begin{definition}
    An $H$-valued random variable $X$ is called {\em Gaussian} if $\langle X,\phi\rangle$ is a real-valued Gaussian random variable, for all $\phi\in H$.
\end{definition}
In what follows, for every $H,K$ separable Hilbert spaces, we  denote by $\mathcal{L}(H,K)$
the space of bounded linear operators from $H$ to $K$.
If $H=K$, we  use the notation $\mathcal{L}(H)$. We  denote by $\mathcal{L}_2(H,K)$ the subspace of all operators $Q \in\,\mathcal{L}(H,K)$ such that
\[\|Q\|_{2}=\left(\,\sum_{i=1}^\infty \|Q e_i\|_K^2\right)^{1/2}<\infty,\]
where $\{e_i\}_{i\in\,\mathbb{N}}$ is any orthonormal basis of $H$. If this series converges, the value of $\|Q\|_2$ is independent of the chosen basis. Operators in $\mathcal{L}_2(H,K)$ are called Hilbert-Schmidt operators on $H$ with values in $K$. 
Next, we denote by $\mathcal{L}^+(H)\subset\mathcal{L}(H)$ consisting of all non-negative and self-adjoint operators.  
For every $Q \in\,\mathcal{L}^+(H)$, the trace of $Q$ is defined by
\begin{align*}
    \mbox{Tr}\,Q=\sum_{i=1}^\infty\,\langle Qe_i,e_i\rangle,
\end{align*}
where $\{e_i\}_{i\in\,\mathbb{N}}$ is any orthonormal basis of $H$. Also in this case, if this series converges, the value of $\mbox{Tr}\,Q$ is independent of the chosen basis, and we say that $Q$ belongs to $\mathcal{L}^+_1(H)$, the space of {\em trace-class} operators on $H$. 
Finally, we would like to remark that if $Q,R \in\,\mathcal{L}_2(H,K)$, then $R^\star Q \in\,\mathcal{L}^+_1(H)$ and
\[\langle R,Q\rangle _{\mathcal{L}_2(H,K)}=\mbox{Tr} \,(R^\star Q)=\mbox{Tr} \,(Q^\star R),\]
defines an inner product in $\mathcal{L}_2(H,K)$, whose associated norm is the norm $\|\cdot\|_2$.

The following result summarizes standard properties of covariance operators associated with square-integrable random variables.
\begin{lemma}\label{lem_finite_trace}
    Let $X$ be a random variable in $L^2(\Omega,H)$, with mean $\mu \in\,H$ and covariance operator $Q \in\,\mathcal{L}(H)$. Then, the operator $Q$ belongs to $\mathcal{L}^+_1(H)$, with
    \begin{align*}
        \mbox{{\em Tr}}\,Q=\mathbb E\,\|X-\mu\|^2.
    \end{align*}
\end{lemma}
We next introduce the  {\em cylindrical Wiener process}, which will provide the driving noise for the stochastic differential equations considered in this paper.
\begin{definition}\label{def2.4}
Let  $\{e_i\}_{i \in\,\mathbb{N}}$ be an orthonormal basis of the separable Hilbert space $H$ and let $\{\beta_i(t)\}_{i \in\,\mathbb{N}}$ be a sequence of mutually independent standard Brownian motions, all defined on the stochastic basis $(\Omega, \mathcal{F}, \{\mathcal{F}_t\}_{t\geq 0}, \mathbb{P})$. We define
    \begin{align}\label{series}
        W_t=\sum_{i=1}^\infty\beta_i(t)e_i,\ \ \ \ t\geq 0.
    \end{align}
\end{definition}
It can be easily checked that series above converges in $L^2(\Omega;K)$ for any $K$, an Hilbert space such that the embedding $H\to K$ belongs to the space of Hilbert-Schmidt operators $\mathcal{L}_2(H,K)$. Moreover,
\[\|W_t\|^2_{L^2(\Omega,K)}=t\,\sum_{i=1}^\infty \|e_i\|_K^2.\]



\subsection{Differentiation and Expansion in Hilbert Spaces}

Since we are evaluating functionals over function spaces, the notions of Fr\'echet differentiation and Taylor expansion in Hilbert spaces are essential. We review these concepts and unify notations, following~\cite{cartan2012differential, Prato_Zabczyk_2014}. Throughout the paper, when the context is clear, we do not distinguish functionals and functions.

Let $E$ and $F$ be Banach spaces, and let $E_o\subset E$ be an open set. 
\begin{definition}
    A function $f:E_o\to F$ is said to be Fr\'echet differentiable at $a\in E_o$ if there exists a bounded linear map $Df(a)\in\mathcal L(E,F)$ such that
    \begin{align}\label{deri_Banach}
        \lim_{\substack{h\to0\\a+h\in E_o}}\frac{\|f(a+h)-f(a)-Df(a)\|_F}{\|h\|_E}=0.
    \end{align}
\end{definition}
In case $E=H$, for some Hilbert space $H$, and $F=\mathbb{R}$, $Df(a)$ is $H$-valued, and we denote it  $\nabla f(a) \in\,H$.

Next, we examine higher order differentiability.
\begin{definition}
A function $f:E_o\to F$ is twice differentiable at $a\in U$ if $Df:E_o\to \mathcal{L}(E,F)$ is differentiable at $a$. The second derivative $D^2f(a)$ is the derivative of $Df$ at $a$ and belongs to $\mathcal{L}(E;\mathcal{L}(E,F))=\mathcal{L}^2(E,F)$. Higher-order derivatives are defined inductively, with the $n$-th derivative $D^n f(a)$ belonging to $\mathcal L^n(E; F)$.
\end{definition}
Note that according to this definition, for every $n \in\,\mathbb{N}$, $\mathcal L^n(E;F)$ is the space of continuous $n$-linear maps from $E^n=E\times \cdots\times E$ into $F$. Akin to directional derivative, we also have high-order Gateaux derivatives as well. For $h_1,\ldots,h_n\in E$, we write
\begin{align*}
    D^nf(a)\,(h_1,\ldots,h_n)\in F\,,
\end{align*}
as the evaluation of the multilinear map $D^nf(a)$ at $(h_1,\ldots,h_n) \in\,E^n$. When $h_1=h_2=\cdots=h_n$, we condense the expression to be $D^nf(a)\,(h)$.

These notations are useful to express functional Taylor expansion below. This formula will be shown to be of core use in the weak error analysis in \Cref{Sec_anal}.

\begin{theorem}[Taylor Formula with the Integral Remainder]\label{Taylor_exp}
	Let $f:E_o\to F$ be a function with sufficient smoothness, and assume that the line segment $[a,a+h]$ is contained in $E_o$. Then,
\begin{align*}
	   f(a+h)&=f(a)+Df(a)\, h+\frac{1}{2}D^2f(a)\,(h)+\cdots+\frac{1}{n!}D^nf(a)\, (h)\\
        &\hspace{.4cm}+\int_0^1\frac{(1-t)^n}{n!}D^{n+1}f(a+th)\,(h)dt.
	\end{align*}
\end{theorem}

Lastly, we recall a standard well-posedness result for stochastic evolution equations driven by cylindrical Wiener noise introduced in \eqref{series}. Let $U$ and $H$ be separable Hilbert spaces, let $W_t$ be a cylindrical Wiener process on $U$, and consider
\begin{align}\label{wp4}
    dX_t=b(X_t)dt+\sigma(X_t)dW_t,\quad X_0=x\in H.
\end{align}
\begin{theorem}\label{thm_wellposed}
Let $H$ be a separable Hilbert space and consider equation \eqref{wp4}. Assume that 
$b:H\to H$ and $\sigma:H\to\mathcal L_2(U,H)$ are locally Lipschitz continuous, 
and there exists $c>0$ such that
\begin{equation} \label{wp1}
\langle b(x),x\rangle + \|\sigma(x)\|_{\mathcal L_2(H)}^2
\le c\,(1+\|x\|^2),
\qquad\ x\in H.
\end{equation}
Then for every $x\in H$ there exists a unique global adapted continuous solution $X_t$ with $X_0=x$ such that
\begin{equation}\label{wp5}
\mathbb E \sup_{t\in[0,T]} \|X_t\|^{2p}
\le c_p(T)(1+\|x\|^{2p}),\quad\forall~T>0,~~p\ge1.
\end{equation}
\end{theorem}
This is a classical result in the case of finite-dimensional spaces, and we show how it adapts easily to the case of general Hilbert spaces in \Cref{pf_thm_wellposed}, following {the same arguments used e.g. in \cite[Theorem 2.3.5]{MR2380366}}.

\section{Stochastic Gradient Descent and Modified Equations in Infinite Dimensions}\label{Sec_SME}

In this section we follow the strategy of~\cite{li2019stochastic} to derive a stochastic modified equation (SME) for stochastic gradient descent (SGD) in an infinite-dimensional Hilbert space. We begin by formulating the SGD algorithm in this setting and analyzing the statistical properties of the stochastic gradients; see Section~\ref{subsec_sgd}. Based on these properties, we then derive a corresponding continuous-time stochastic evolution equation by matching the mean and covariance of the discrete stochastic gradients; see Section~\ref{subsec_sme}. Finally, we address the well-posedness of the resulting stochastic differential equation. In particular, we identify conditions on the diffusion operator that ensure existence and uniqueness of solutions; see Section~\ref{subsec:wellposedness}.

\subsection{SGD in Hilbert Spaces and Noise Structure}\label{subsec_sgd}
Consider the following optimization problem defined on an infinite-dimensional Hilbert space $H$:
\begin{align}\label{minimiz_problem}
    \min_{\phi\in H}\mathcal F(\phi):=\mathbb E\,\mathcal F_\gamma(\phi),
\end{align}
where $\{\mathcal F_\gamma:\gamma\in\Gamma\}$ is a family of functionals from $H$ to $\mathbb R$, and $\gamma$ is a $\Gamma$-valued random variable representing the source of uncertainty. The expectation is taken with respect to the distribution of $\gamma$.

A standard approach for solving \eqref{minimiz_problem} is gradient descent (GD),
\begin{align}\label{GD}
    \phi_{n+1}=\phi_n-\eta\, \nabla\mathcal F(\phi_n)
    =\phi_n-\eta\,\mathbb{E}\,\nabla\mathcal F_\gamma (\phi_n),
    \qquad n\ge0,
\end{align}
where $\eta > 0$ is the learning rate and $\nabla\mathcal F$ denotes the Fréchet derivative of $\mathcal F$. For the iteration \eqref{GD} to be well-defined, we assume
\begin{center}
(A1) $\mathcal F$ is Fréchet differentiable on $H$,
\qquad
(A2) the initial iterate $\phi_0$ belongs to $H$.
\end{center}

In practice, evaluating the full gradient
\[
\nabla\mathcal F (\phi)=\mathbb{E}\,\nabla\mathcal F_\gamma (\phi)
\]
requires computing an expectation over all realizations of $\gamma$, which can be computationally expensive, or even infeasible. To reduce this cost, one typically replaces the full gradient by a stochastic approximation and considers the stochastic gradient descent (SGD)
\begin{align}\label{SGD}
    \phi_{n+1}=\phi_n-\eta\, \nabla\mathcal F_{\gamma_n}(\phi_n)\,,
    \qquad n\ge0\,,
\end{align}
where $\{\gamma_n\}_{n\ge0}$ is a sequence of i.i.d. random variables with the same distribution as $\gamma$.

To facilitate the derivation of the continuous-time approximation, we rewrite the iteration \eqref{SGD} in the perturbation form
\begin{align}\label{SGD_noise}
\phi_{n+1}
=
\phi_n-\eta\,\nabla\mathcal F(\phi_n)+\eta\,V(\phi_n;\gamma_n)\,,
\end{align}
where the intrinsic noise term is defined by
\begin{align*}
V(\phi;\gamma):=\nabla\mathcal F(\phi)-\nabla\mathcal F_\gamma(\phi).
\end{align*}

By construction, the process $\{V(\phi_n;\gamma_n)\}_{n\ge0}$ satisfies the martingale difference property with respect to the natural filtration $\mathfrak F_n=\sigma(\gamma_0,\dots,\gamma_{n-1})$, namely
\[
\mathbb E\!\left[V(\phi_n;\gamma_n)\mid \mathfrak F_n\right]=0\,,
\qquad n\ge1\,.
\]

The randomness introduced by the stochastic gradients induces fluctuations in the dynamics. To quantify these fluctuations, we characterize the covariance structure of the noise. For every $\phi\in H$, we define the covariance operator $Q(\phi):H\to H$ by
\begin{equation*}
\langle Q(\phi)u,v\rangle
=
\mathrm{Cov}\big(\langle V(\phi;\gamma),u\rangle,\langle V(\phi;\gamma),v\rangle\big),
\qquad \forall u,v\in H.
\end{equation*}
or formally written as:
\[
Q(\phi)=\mathbb E_\gamma\big[V(\phi;\gamma)\otimes V(\phi;\gamma)\big]\,.
\]

It is immediate to check that  $Q(\phi)$ is a bounded and non-negative selfadjoint linear operator in $H$, for every $\phi$ in $H$.
In what follows, we shall assume
\begin{center}
(A3) For every $\phi\in H$, the stochastic gradient has finite second moment
\[
\mathbb E_\gamma\,\|\nabla\mathcal F_\gamma(\phi)\|^2<\infty .
\]
\end{center}
This implies that $Q(\phi) \in\,\mathcal{L}^+_1(H)$, as
\[\mbox{Tr}\,Q(\phi)=\sum_{i=1}^\infty \mathbb{E}_\gamma |\langle V(\phi;\gamma),e_i\rangle|^2=\mathbb{E}_\gamma \| V(\phi;\gamma)\|_H^2\leq \mathbb E_\gamma\,\|\nabla\mathcal F_\gamma(\phi)\|^2.\]
\subsection{Derivation of Infinite-Dimensional SMEs}\label{subsec_sme}

We now derive a stochastic differential equation on the Hilbert space $H$ that approximates the discrete SGD trajectory \eqref{SGD_noise}. Following the diffusion-approximation viewpoint, we seek a continuous-time stochastic evolution whose drift and covariance agree with those of the SGD updates.

We propose the following stochastic modified equation (SME):
\begin{align}\label{sde_model}
    d\varphi_t = -\nabla \mathcal{F}(\varphi_t)\,dt + \sqrt{\eta}\,\sigma(\varphi_t)\,dW_t,
\end{align}
where $W_t$ is the cylindrical Wiener process introduced in Definition~\ref{def2.4}, and $\sigma$ is to be defined. It is worth noting that the space of $\sigma$ and $W_t$ are not fixed yet, and they will be selected to ensure the consistency between~\eqref{sde_model} and~\eqref{SGD_noise}.

To relate \eqref{sde_model} to SGD, we discretize it by the Euler--Maruyama scheme with step size $\eta$:
\begin{align}\label{EM_scheme}
    \tilde{\varphi}_{n+1}
    =
    \tilde{\varphi}_n
    -
    \eta\,\nabla \mathcal{F}(\tilde{\varphi}_n)
    +
    \sqrt{\eta}\,\sigma(\tilde{\varphi}_n)\bigl(W_{(n+1)\eta}-W_{n\eta}\bigr).
\end{align}

Comparing \eqref{EM_scheme} with the perturbation form of SGD in \eqref{SGD_noise}, we see immediately that the drift terms coincide, and that both noise terms have mean zero. It remains to compare their covariance structures. For any $u,v\in H$,
\begin{align*}
&\mathrm{Cov}\Bigl(
\bigl\langle \sqrt{\eta}\,\sigma(\phi)\bigl(W_{(n+1)\eta}-W_{n\eta}\bigr),u\bigr\rangle,
\bigl\langle \sqrt{\eta}\,\sigma(\phi)\bigl(W_{(n+1)\eta}-W_{n\eta}\bigr),v\bigr\rangle
\Bigr) \\
&\qquad
=
\eta\,
\mathbb E\Big[
\bigl\langle W_{(n+1)\eta}-W_{n\eta},\sigma(\phi)^*u\bigr\rangle
\bigl\langle W_{(n+1)\eta}-W_{n\eta},\sigma(\phi)^*v\bigr\rangle
\Big] \\
&\qquad
=
\eta^2 \langle \sigma(\phi)\sigma(\phi)^*u,v\rangle\,.
\end{align*}
Therefore, in order for the Euler--Maruyama discretization of the SME~\eqref{EM_scheme} to match the first two moments of the SGD increment~\eqref{SGD_noise}, we need:
\begin{align}\label{cov_match}
    \mathrm{Cov}\Bigl(
    \sqrt{\eta}\,\sigma(\phi)\bigl(W_{(n+1)\eta}-W_{n\eta}\bigr)
    \Bigr)
    =\eta^2\sigma(\phi)\sigma(\phi)^*=
    \eta^2 Q(\phi)
    =
    \mathrm{Cov}\bigl(\eta V_n(\phi)\bigr).
\end{align}
Since the randomness in \eqref{SGD_noise} arises from sampling $\gamma\in\Gamma$, whereas the SME is driven by Gaussian noise, agreement at the level of these statistical quantities is the natural notion of approximation.

Now we aim to make the definition of $\sigma$ and $W_t$ explicit that satisfies \eqref{cov_match}. To do so, we recall the definition in~\eqref{eqn:L_omega_space}
and define $U:=L_\Gamma^2$.
The diffusion kernel $\sigma:U\to H$ becomes:
\begin{align}\label{sigma_def}
    \sigma(\phi)u:=\mathbb E_\gamma\,\big(V(\phi;\gamma)u(\gamma)\big),\ \ \ \ \ \ u\in U.
\end{align}
This also suggests that $W_t$ is $U$-valued. In the following lemma, we clarify that this construction is appropriate and satisfy~\eqref{cov_match}. Throughout the section, unless otherwise specified, vector norms and inner products are taken in $H$. Norms on operator spaces, such as $\mathcal L_2(U,H)$ or $U$, are indicated explicitly.

\begin{lemma}
    Let $\sigma$ be defined as in \eqref{sigma_def}. Then, $\sigma(\phi)$ satisfies \eqref{cov_match} and belongs to $\mathcal{L}_2(U,H)$, for each $\phi$.
\end{lemma}
\begin{proof}
    Note that for each $h\in H$,
\begin{align*}  \langle\sigma(\phi)u,h\rangle&=\left\langle\mathbb E\,\big(V(\phi;\gamma)u(\gamma)\big),h\right\rangle=\mathbb E\,\big(\left\langle V(\phi;\gamma),h\right\rangle u(\gamma)\big)=\big\langle\langle V(\phi;\gamma),h\rangle,u\big\rangle_U\,,
\end{align*}
which indicates:
\begin{align*}
    \big(\sigma(\phi)^*h\big)(\gamma)=\langle V(\phi;\gamma),h\rangle.
\end{align*}
Thus, one can check that $\sigma\sigma^*=Q$ as follows
\begin{align*}
\sigma(\phi)\sigma(\phi)^*h&=\mathbb E_\gamma\,\big(V(\phi;\gamma)\langle V(\phi;\gamma),h\rangle\big)=Q(\phi)h.
\end{align*}
This definition naturally leads to the fact that $\sigma$ is a Hilbert-Schmidt operator. Denote $(e_k)$ an orthonormal basis of $U$, we have
\begin{align*}
    \langle\sigma(\phi)e_k,h\rangle&=\mathbb E_\gamma\,\big(\langle V(\phi,\gamma),h\rangle e_k(\gamma)\big)=\big\langle\langle V(\phi;\gamma),h\rangle,e_k\big\rangle_U,
\end{align*}
and this implies
\begin{align*}
    \sum_k|\langle\sigma(\phi)e_k,h\rangle|^2&=\sum_k|\big\langle\langle V(\phi;\gamma),h\rangle,e_k\big\rangle_U|^2=\|\langle V(\phi;\gamma),h\rangle\|_U^2
\end{align*}
For an orthonormal basis $(f_j)$ of $H$, we get
\begin{align*}
    \|\sigma(\phi)\|_{\mathcal{L}_2(U,H)}^2&=\sum_k\|\sigma(\phi)e_k\|^2=\sum_k\sum_j|\langle\sigma(\phi)e_k,f_j\rangle|^2=\sum_j\|\langle V(\phi;\gamma),f_j\rangle\|_U^2\\
    &=\sum_j\mathbb E\big(\langle V(\phi;\gamma),f_j\rangle^2\big)=\mathbb E\|V(\phi;\gamma)\|^2<\infty,
\end{align*}
where the last inequality comes from (A3). This ends the proof.
\end{proof}
\subsection{Existence and Uniqueness of Strong Solution to SME}\label{subsec:wellposedness}

With the stochastic modified equation formally derived in \eqref{sde_model}, it remains to verify that the equation is well-posed. Recalling \Cref{thm_wellposed}, existence and uniqueness of solutions follow provided the drift and diffusion coefficients satisfy appropriate local Lipschitz and growth conditions. To this end we introduce the following assumption.
\begin{enumerate}[label=(A\arabic*)]
    \setcounter{enumi}{3}  
    \item There exist random variables $C_L(\gamma)\in L^2(\Gamma)$ and $C_G(\gamma)\in L^1(\Gamma)$ such that for each $\phi,\psi\in H$,
    \begin{align*}
        \|\nabla\mathcal F_\gamma(\phi)-\nabla\mathcal F_\gamma(\psi)\|\le C_L(\gamma)\|\phi-\psi\|,\quad\langle\nabla\mathcal F_\gamma(\phi),\phi\rangle\le C_G(\gamma)(1+\|\phi\|^2)\,.
    \end{align*}
\end{enumerate}
Under this assumption the SME~\eqref{sde_model} is well-posed.
{
\begin{theorem}
    Assume \textnormal{(A4)}. Then for every initial condition $\phi_0\in H$, equation \eqref{sde_model}  with $\sigma$ defined as \eqref{sigma_def} admits a unique adapted strong solution
\[
\phi_t
=
\phi_0
-
\int_0^t \nabla \mathcal{F}(\phi_s)\,ds
+
\sqrt{\eta}\int_0^t \sigma(\phi_s)\,dW(s),
\qquad t\ge0 ,
\]
such that
\[
\phi \in L^p(\Omega;C([0,T];H))
\]
for every $p\ge1$ and $T>0$. Moreover,
\[
\mathbb{E}\,\sup_{t\in[0,T]}\|\phi_t\|^p
\le
c_p(T)\big(1+\|\phi_0\|^p\big).
\]
\end{theorem}
\begin{proof}
    By \Cref{thm_wellposed}, it is enough to show the local Lipschitz continuity and linear growth for both $\nabla\mathcal F$ and $\sigma$ in \eqref{sde_model}.

{\em Condition for $\nabla\mathcal F$.} We can compute the Lipschitz continuity and linear growth of $\nabla\mathcal F$
\begin{align*}
    \|\nabla\mathcal F(\phi)-\nabla\mathcal F(\psi)\|\le\mathbb E\|\nabla\mathcal F_\gamma(\phi)-\nabla\mathcal F_\gamma(\psi)\|\le\mathbb E C_L(\gamma)\,\|\phi-\psi\|,
\end{align*}
and
\begin{align*}
    \langle\nabla\mathcal F(\phi),\phi\rangle=\mathbb E\langle\nabla\mathcal F_\gamma(\phi),\phi\rangle\le \mathbb E C_G(\gamma)\,(1+\|\phi\|^2).
\end{align*}
Noting that $|\mathbb E C_L(\gamma)|^2\leq \mathbb E C^2_L(\gamma)<\infty$ by Jensen,  and by \textnormal{(A4)}, conditions for the velocity field are satisfied in \Cref{thm_wellposed}.

{\em Condition for $\sigma$.} For any orthonormal basis $(e_k)$ and $(f_j)$ for $U$ and $H$, respectively, we use Parseval's identity to derive
\begin{align*}
    \|\sigma(\phi)-\sigma(\psi)\|_{\mathcal{L}_2(U,H)}^2&=\sum_k\|\big(\sigma(\phi)-\sigma(\psi)\big)e_k\|^2=\sum_k\big\|\mathbb E\big[\big(V(\phi;\gamma)-V(\psi;\gamma)\big)e_k(\gamma)\big]\big\|^2\\
    &=\sum_{k,j}\left\langle\mathbb E\,\big(\big(V(\phi;\gamma)-V(\psi;\gamma)\big)e_k(\gamma)\big),f_j\right\rangle^2\\
    &=\sum_{k,j}\left(\mathbb E\,\big(\langle V(\phi;\gamma)-V(\psi;\gamma),f_j\rangle e_k(\gamma)\big)\right)^2\\
    &=\sum_j\mathbb E\,|\langle V(\phi;\gamma)-V(\psi;\gamma),f_j\rangle|^2=\mathbb E\|V(\phi;\gamma)-V(\psi;\gamma)\|^2.
\end{align*}
Using the above relation, one can derive:
\begin{align*}
    \|\sigma(\phi)-\sigma(\psi)\|_{\mathcal{L}_2(U,H)}^2&=\mathbb E\|V(\phi;\gamma)-V(\psi;\gamma)\|^2\\
    &=\mathbb E\|\nabla\mathcal F_\gamma(\phi)-\nabla\mathcal F_\gamma(\psi)\|^2-\left(\mathbb E\|\nabla\mathcal F_\gamma(\phi)-\nabla\mathcal F_\gamma(\psi)\|\right)^2\\
    &\le\mathbb E\|\nabla\mathcal F_\gamma(\phi)-\nabla\mathcal F_\gamma(\psi)\|^2\le\mathbb E C_L(\gamma)^2\,\|\phi-\psi\|^2,
\end{align*}
with \textnormal{(A4)}, the Lipschitz continuity of $\sigma$ is confirmed.
Now, it only remains to check the linear growth of $\sigma$. Through Jensen's inequality,
\begin{align}\label{sigma_bound}
    \|\sigma(\phi)\|_{\mathcal{L}_2(U,H)}^2=\mathbb E\|V(\phi;\gamma)\|^2\le2\|\nabla\mathcal F(\phi)\|^2+2\mathbb E\|\nabla\mathcal F_\gamma(\phi)\|^2\le4\mathbb E\|\nabla\mathcal F_\gamma(\phi)\|^2.
\end{align}
Note that
\begin{align*}
    \mathbb E\|\nabla\mathcal F_\gamma(\phi)\|^2&\le2\mathbb E\|\nabla\mathcal F_\gamma(\phi)-\nabla\mathcal F(0)\|^2+2\|\nabla\mathcal F(0)\|^2\\
    &\le2\,\mathbb E C_L(\gamma)^2\,\|\phi\|^2+2\mathbb E\|\nabla\mathcal F_\gamma(0)\|^2\lesssim(1+\|\phi\|^2),
\end{align*}
where we used \textnormal{(A4)}. Plugging this back in~\eqref{sigma_bound} we justify the condition for $\sigma$ being satisfied in~\Cref{thm_wellposed}, and hence complete the proof.
\end{proof}
}

\section{Weak Error Analysis}\label{Sec_anal}
In this section, we analyze the weak error between the SGD iteration and the discretized SME. Our main objective is to establish a second-order weak approximation result, showing that the discrepancy between the two dynamics is of order $\mathcal O(\eta^2)$ on finite time intervals. This justifies the use of SME as a mean to understand SGD solver.

We rewrite $\phi_n$ as the SGD iterate and $\tilde\varphi_n$ as the discretized SME:
\begin{subequations}
\begin{align}
    \phi_{n+1} &= \phi_n + \eta b(\phi_n) + \eta V(\phi_n), \label{iter_disc_a}\\
    \tilde\varphi_{n+1} &= \tilde\varphi_n + \eta b(\tilde\varphi_n) + \eta \sigma(\tilde\varphi_n)\Delta W \label{iter_disc_b}
\end{align}
\end{subequations}

where  $\sigma$ is chosen as in \Cref{Sec_SME}, so that for every $\phi \in\,H$ the covariance of $\sigma(\phi)\Delta W$
matches that of the stochastic gradient noise $V(\phi)$.

To quantify the weak error, we introduce the class of test functionals $\mathcal G$ defined by
\begin{align*}
  \mathcal G:=\{g: H\to\mathbb R~\mid~\|D^i\,g\|_\infty:=\sup_{\phi \in\,H}\|D^i g(\phi)\|_{\mathcal{L}^i(H)}<\infty,\quad i=1,2,3\},
\end{align*}
where $D^i g$ denotes the $i$-th Fr\'echet derivative of $g$. Throughout the section, when the expectation $\mathbb{E}$ appears, the expectation is taken with respect to both $\phi_n$ and $\tilde{\varphi}_n$, in both $\gamma_n$ for $\Delta W$ space.

We first present the theorem in an informal manner:
\begin{theorem}[informal statement]\label{thm_weak_conv_informal}
    Suppose that $\phi_n$ and $\tilde\varphi_n$ are stochastic iterative processes starting from a fixed initial state $\phi_0\in H$ generated by the first and second stochastic iterations in \eqref{iter_disc_a}-\eqref{iter_disc_b}, respectively. Assume that $b$ and $\sigma$ satisfy certain conditions and let $g$ be a smooth functional, then
    \begin{align*}
        \max_{k=0,\cdots,\lfloor T/\eta\rfloor}\left|\mathbb E\big(g(\tilde\varphi_k)-g(\phi_k)\big)\right|=\mathcal{O}(\eta^2)\,.
    \end{align*}
\end{theorem}
The statement and the conditions are to be made more precise in Theorem~\ref{thm_weak_conv}. At first glance, the result may appear somewhat counterintuitive. Indeed, for a stochastic differential equation, the Euler–Maruyama discretization typically achieves only first-order accuracy in the weak sense with respect to the time step $\eta$. We emphasize, however, that the two quantities compared in the theorem are $\phi_k$ and $\varphi_k$. In particular, we are not comparing a continuous-time SDE solution with its discrete-time approximation. Instead, both $\phi_k$ and $\varphi_k$ arise from discrete dynamics. The result therefore shows that these two discrete processes approximate each other to second order in the step size.

The main strategy of analyzing error comes down to induction. To do so, we first unify two notations:
\begin{itemize}
    \item $\phi_n(\phi,k)$ denotes solution to~\eqref{iter_disc_a} at $t_n=n\eta$ with the initial condition $\phi_k=\phi$ placed at $t_k$;
    \item $\widetilde\varphi_n(\phi,k)$ denotes solution to~\eqref{iter_disc_b} at $t_n=n\eta$ with the initial condition $\widetilde\varphi_k=\phi$ placed at $t_k$.
\end{itemize}

Now we follow~\cite{li2019stochastic} and propagate solutions:
\begin{align*}
    	\mathbb E\,g(\tilde\varphi_k)&=\mathbb E\,g(\tilde\varphi_k(\phi_0,0))\\
        &=\mathbb E\,g\Big(\tilde\varphi_k\big(\tilde\varphi_1(\phi_{0},0),1\big)\Big)-\mathbb E\,g\Big(\tilde\varphi_k(\phi_1,1)\Big)+\mathbb E\,g\Big(\tilde\varphi_k(\phi_1,1)\Big)\\
        &=\mathbb E\,g\Big(\tilde\varphi_k\big(\tilde\varphi_1(\phi_{0},0),1\big)\Big)-\mathbb E\,g\Big(\tilde\varphi_k(\phi_1,1)\Big)+\mathbb E\,g\Big(\tilde\varphi_k\big(\tilde\varphi_2(\phi_{1},1),2\big)\Big)-\mathbb E\,g\Big(\tilde\varphi_k(\phi_2,2)\Big)+\mathbb E\,g\Big(\tilde\varphi_k(\phi_2,2)\Big)\\
        &=\cdots\\
        &=\sum_{l=1}^{k-1}\left(\mathbb E\,g\Big(\tilde\varphi_k\big(\tilde\varphi_l(\phi_{l-1},l-1),l\big)\Big)-\mathbb E\,g\Big(\tilde\varphi_k(\phi_l,l)\Big)\right)+\mathbb E\,g\Big(\tilde\varphi_k(\phi_{k-1},k-1)\Big)\,.
    \end{align*}
    As a consequence,
    \begin{align}
    	\mathbb E\,g(\tilde\varphi_k)-\mathbb E\,g(\phi_k)&=\sum_{l=1}^{k-1}\Big(\mathbb E\,g\big(\tilde\varphi_k\big(\tilde\varphi_l(\phi_{l-1},l-1),l\big)\big)-\mathbb E\,g\big(\tilde\varphi_k(\phi_l,l)\big)\nonumber\\
    	&\hspace{.4cm}+\mathbb E\,g\big(\tilde\varphi_k(\phi_{k-1},k-1)\big)-\mathbb E\,g\big(\phi_k(\phi_{k-1},k-1)\big)\Big)\nonumber\\
        &=\sum_{l=1}^{k-1}G_l+G_k\,.\label{eqn:induction}
    \end{align}
    For $G_k$, this is to use $g$ to test $\tilde\varphi_k(\phi_{k-1},k-1) - \phi_k(\phi_{k-1},k-1)$. For each $G_l$, reckoning $\phi_l=\phi_l(\phi_{l-1},l-1)$, this term comes down to using $g\big(\tilde\varphi_k(\cdot,l))$ to test $\tilde\varphi_l(\phi_{l-1},l-1) - \phi_l(\phi_{l-1},l-1)$. This clearly suggests two components to be analyzed:
    \begin{itemize}
        \item[$\bullet$] One step error: $\tilde\varphi_l(\phi_{l-1},l-1) - \phi_l(\phi_{l-1},l-1)$ needs to be small for all $l\leq k$.
        \item[$\bullet$] Test function regularity: $g$ and $g\big(\tilde\varphi_k(\cdot,l))$ have to belong to the same test function class with the same regularity.
    \end{itemize}
    We prepare these two pieces in Section~\ref{sec:preparation_proof} and sum them up for the full proof in Section~\ref{sec:global_proof}.

\subsection{One step error and propagation of regularity}\label{sec:preparation_proof}

Throughout this section, since all results are local in time with iteration conducted over one step, we suppress the time index $n$ and denote the stochastic gradient noise $V_n$ by $V$ for notational convenience. We emphasize that $V$ is a random mapping due to its dependence on the random variable $\gamma$, that is, $V(\phi_n)$ represents a stochastic noise. We also suppress the index $n$ in  $\Delta W_n$, the normalized increment of the cylindrical Wiener process $W_t$, and denote it by $\Delta W$.


The two main pieces for the proof are: 1. analyzing one-step error; and 2. preserving regularity of the test function along iteration. The two results are summarized in the following two lemmas.
\begin{lemma}[One-step Error Estimate]\label{lem_one_step}
    Let $\phi_1$ and $\tilde{\varphi}_1$ be the solutions to \eqref{iter_disc_a}-\eqref{iter_disc_b}, starting from the same deterministic data $\phi_0 = \tilde{\varphi}_0 = \phi$. Then, for any $g \in \mathcal{G}$,
    \begin{align*}
        \big| \mathbb{E}\,\left(g(\phi_1) - g(\tilde{\varphi}_1)\right) \big| 
        \le \eta^3 \|D^3 g\|_\infty \, \mathbb{E}\,\left( \|\sigma(\phi)(\Delta W)\|^3 + \|V(\phi)\|^3 \right).
    \end{align*}
\end{lemma}
\begin{proof}
    One uses the Taylor expansion to $g(\tilde\varphi_1)$ and $g(\phi_1)$ to derive
    \begin{align*}
        \mathbb E\,\left(g(\tilde\varphi_1)-g(\phi_1)\right)&=\mathbb E\left(g\Big(\phi+\eta b(\phi)+\eta\sigma(\phi)(\Delta W)\Big)-g\Big(\phi+\eta b(\phi)+\eta V(\phi)\Big)\right)\\[10pt]
        &=\mathbb E\left(\eta Dg(\phi)\,(\sigma(\phi)(\Delta W)-V(\phi))+\eta^2D^2 g(\phi)[(\sigma(\phi)(\Delta W))^2-(V(\phi))^2]\right)\\[10pt]
        &\hspace{.4cm}+\eta^3\mathbb E\left(D^3 g(\hat\varphi(\phi,\omega))\,(\sigma(\phi)(\Delta W))^3-D^3 g(\hat\phi(\phi,\omega))\,(V(\phi))^3\right)
    \end{align*}
    for some $\hat\varphi(\phi,\omega),\hat\phi(\phi,\omega)\in H$ depending on the sample $\omega\in\Omega$. Since $\sigma(\phi)$ is set to be the same in the expectation and covariance with $V(\phi)$, the first and second order terms are cancelled. Thus, we use the boundedness of the third derivative in $g$ to obtain the desired result.
\end{proof}

The second component needed for our proof is the propagation of regularity of the test functional, so that the functional space $\mathcal{G}$ is invariant along the iterate. Formally, for every $g$ we define the functional $A(g)$ as:
\begin{align}\label{oper_A}
    A(g)(\cdot): H\to\mathbb{R}\,\quad\text{with}\quad A(g)(\phi):= \mathbb{E}\, g(\tilde{\varphi}_1(\phi,0))\,,
\end{align}
and thus $A$ is an operator that takes values in $g\in\mathcal{G}$ to produce another functional. We are to show $A$ is invariant in $\mathcal{G}$ in the sense that $A(g)\in\mathcal{G}$ also. To quantify them, we define some quantities:
\begin{align}\label{eqn:def_C_i}
    C_i(b,\sigma,\eta)&=(a_i+\eta\|D^ib\|_\infty)^2+\eta^2\|D^i\sigma\|_\infty^2,\quad i=1,2,3,\\[8pt]
    C_4(b,\sigma,\eta)&=(1+\eta\|Db\|_\infty)^4+6\eta^2(1+\eta\|Db\|_\infty)^2\|D\sigma\|_\infty^2+3\eta^4\|D\sigma\|_\infty^4\,,\nonumber
\end{align}
for the constants $a_1=1$ and $a_2=a_3=0$. The propagation of regularity is summarized in the following lemma.

\begin{lemma}\label{lem_ug}
    Suppose that $g\in\mathcal G$ and define a functional $u=A(g)$ according to~\eqref{oper_A}, then:
    \begin{align*}
        u(\phi):=\mathbb E\,g(\tilde\varphi_1(\phi,0))=\mathbb E\,\left(g\left(\phi+\eta b(\phi)+\eta\sigma(\phi)(\Delta W)\right)\right).
    \end{align*}
    Then, we have $u\in\mathcal G$. In particular, we have the following control over derivatives:
    \begin{align*}
        \|Du\|_\infty&\le\sqrt{C_1}\|Dg\|_\infty\,,\\[10pt]
        \|D^2 u\|_\infty&\le C_1\|D^2 g\|_\infty+\sqrt{C_2}\|Dg\|_\infty\,,\\[10pt]
        \|D^3 u\|_\infty&\le C_4\|D^3 g\|_\infty+\sqrt{C_1C_2}\|D^2 g\|_\infty+\sqrt{C_3}\|Dg\|_\infty\,.
    \end{align*}
\end{lemma}
\begin{proof}
    We split the proof into two parts.\\
    
    \noindent {\em Step 1.} Here, we estimate the upper bound for the norm of Fr\'echet derivatives of the function $\tilde\varphi_1(\cdot,0)$ in the expectation sense:
    \begin{align*}
        \mathbb E\,\|D^i \tilde\varphi_1(\phi)(h)\|^2\quad\mbox{and}\quad\mathbb E\,\Big(\|D\tilde\varphi_1(\phi)h\|\|D\tilde\varphi_1(\phi)k\|\|D\tilde\varphi_1(\phi)l\|\Big).
    \end{align*}
    Theses are required to obtain the upper bound of $u^{(i)}$ in {\em Step 2}. One can apply the chain rule to $\tilde\varphi_1$ to see, according to definition~\eqref{deri_Banach}:
    \begin{align*}
        &D \tilde\varphi_1(\phi)h=h+\eta Db(\phi)h+\eta\,D\sigma(\phi)h\,(\Delta W)\\[10pt]
        &D^2 \tilde\varphi_1(\phi)(h,k)=\eta D^2b(\phi)(h,k)+\eta D^2\sigma(\phi)(h,k)\,(\Delta W),\\[10pt]
        &D^3 \tilde\varphi_1(\phi)(h,k,l)=\eta D^3b(\phi)(h,k,l)+\eta D^3\sigma(\phi)(h,k,l),
    \end{align*}
    for all $h,k,l\in H$. Due to the regularity of $b$ and $\sigma$, we obtain
    \[\mathbb E\|D^i\tilde\varphi_1(\phi)(h_1,\cdots,h_i)\|^2\le C_i(b,\sigma,\eta)\|h_1\|^2\cdots\|h_i\|^2,\quad\forall~i=1,2,3,\]
    where we call the definition in~\eqref{eqn:def_C_i}.
    
    On the other hand, denote
    \begin{align*}
        D\tilde\varphi_1(\phi)h=:U_h+V_h(\Delta W),
    \end{align*}
    which satisfies
    \begin{align}\label{UV}
        \|U_h\|\le\|h\|(1+\eta\|Db\|_\infty),\quad\|V_h\|\le\eta\|D\sigma\|_\infty\|h\|.
    \end{align}
    Hence, one has
    \begin{align}\label{E4}
        \begin{aligned}
            \mathbb E\|D\tilde\varphi_1(\phi)h\|^4&=\mathbb E\Big(\|U_h\|^2+2\langle U_h,V_h(\Delta W)\rangle+\|V_h(\Delta W)\|^2\Big)^2\\[8pt]
            &\le\|U_h\|^4+6\|U_h\|^2\mathbb E\|V_h(\Delta W)\|^2+\mathbb E\|V_h(\Delta W)\|^4\\[8pt]
            &\le\|U_h\|^4+6\|U_h\|^2\|V_h\|^2+3\|V_h\|^4.
        \end{aligned}
    \end{align}
    By combining \eqref{UV} and \eqref{E4}, one gets
    \begin{align}\label{E4h}
        \mathbb E\|D\tilde\varphi_1(\phi)h\|^4\le C_4(b,\sigma,\eta)\|h\|^4,
    \end{align}
    where $C_4(b,\sigma,\eta)$ is defined in \eqref{eqn:def_C_i}.
    
    Now, we use H\"older's inequality, the Lyapunov inequality and the relation \eqref{E4h} to obtain
    \begin{align*}
        & \mathbb E\,\Big(\|D\tilde\varphi_1(\phi)h\|\|D\tilde\varphi_1(\phi)k\|\|D\tilde\varphi_1(\phi)l\|\Big)\\[8pt]
        &\hspace{1cm}\le\left(\mathbb E\|D\tilde\varphi_1(\phi)h\|^3\right)^{1/3}\left(\mathbb E\|D\tilde\varphi_1(\phi)k\|^3\right)^{1/3}\left(\mathbb E\|D\tilde\varphi_1(\phi)l\|^3\right)^{1/3}\\[8pt]
        &\hspace{1cm}\le\left(\mathbb E\|D\tilde\varphi_1(\phi)h\|^4\right)^{1/4}\left(\mathbb E\|D\tilde\varphi_1(\phi)k\|^4\right)^{1/4}\left(\mathbb E\|D\tilde\varphi_1(\phi)l\|^4\right)^{1/4}\le C_4(b,\sigma,\eta)\|h\|\|k\|\|l\|.
    \end{align*}
    
    \noindent {\em Step 2.} Now, we compute the upper bound of Fr\'echet derivatives of $u$. Using the chain rule, one gets
    \begin{align*}
        |Du(\phi)h|&=\left|\mathbb E\Big(Dg(\tilde\varphi_1(\phi))D\tilde\varphi_1(\phi)h\Big)\right|\le\mathbb E\Big(\|Dg(\tilde\varphi_1(\phi))\|\|D\tilde\varphi_1(\phi)h\|\Big)\\[8pt]
        &\le\sqrt{C_1}\|Dg\|_\infty\|h\|\,,
    \end{align*}
    where we use the boundedness of $Dg$ and the result in {\em Step 1}. The estimate for the second order derivative is similar as follows
    \begin{align*}
        D^2 u(\phi)(h,k)&=\mathbb E\Big(D^2g(\tilde\varphi_1(\phi))\big(D\tilde\varphi_1(\phi)h,D\tilde\varphi_1(\phi)k\big)+Dg(\tilde\varphi_1(\phi)) D^2\tilde\varphi_1(\phi)(h,k)\Big)\\[8pt]
        &\le\|D^2g\|_\infty\mathbb E\Big(\|D\tilde\varphi_1(\phi)h\|\|D\tilde\varphi_1(\phi)k\|\Big)+\|Dg\|_\infty\mathbb E \|D^2\tilde\varphi_1(\phi)(h,k)\|\\[8pt]
        &\le C_1(b,\sigma,\eta)\|D^2g\|_\infty\|h\|\|k\|+C_2(b,\sigma,\eta)\|Dg\|_\infty\|h\|\|k\|,
    \end{align*}
    where the last inequality from the H\"older's inequality and the result in {\em Step 1}. Lastly, the third order derivative can be estimated by the similar method as first and second order derivative. For readers' convenience, we provide the explicit form of the third derivative
    \begin{align*}
        D^3u(\phi) (h,k,l)&=\mathbb E\Big(D^3 g(\tilde\varphi_1(\phi))\left(D\tilde\varphi_1(\phi)h,D\tilde\varphi_1(\phi)k,D\tilde\varphi_1(\phi)l\right)\\[8pt]
        &\hspace{.4cm}+D^2g(\tilde\varphi_1(\phi))\big(D^2\tilde\varphi_1(\phi)(h,k),D\tilde\varphi_1(\phi)l\big)+D^2g(\tilde\varphi_1(\phi))\big(D^2\tilde\varphi_1(\phi)(h,l),D\tilde\varphi_1(\phi)k\big)\\[8pt]
        &\hspace{.4cm}+D^2g(\tilde\varphi_1(\phi))\big(D^2\tilde\varphi_1(\phi)(k,l),D\tilde\varphi_1(\phi)h\big)+Dg(\tilde\varphi_1(\phi))D^3\tilde\varphi_1(\phi)(h,k,l)\Big).
    \end{align*}
    We omit the detail and complete the proof.
\end{proof}

An immediate corollary follows. Through induction, we let
\begin{align*}
    g_{n+1} = A(g_n), \quad  n \ge 0\,.
\end{align*}
Then clearly all $g_n\in\mathcal{G}$, and we summarize an explicit bound.
\begin{corollary}\label{cor_gn3bound}
    Suppose that $g_n\in\mathcal G$, for $n\ge1$, are generated by $A$ starting from $g_0\in\mathcal G$. Then, we have the upper bound of $\|D^3g_n\|_\infty$ as follows
    \begin{align*}
        \|D^3g_n\|_\infty\le A_1C_4^n+A_2C_1^n+A_3C_1^{n/2},
    \end{align*}
    where $C_i=C_i(b,\sigma,\eta)$ were defined in~\eqref{eqn:def_C_i} and $A_i=A_i(g,b,\sigma,\eta)$ are:
    \begin{align*}
        A_1(g,b,\sigma,\eta)&=\|D^3g_0\|_\infty-A_2-A_3,\\[8pt]
        A_2(g,b,\sigma,\eta)&=\left(\|D^2g_0\|_\infty+\frac{\|Dg_0\|_\infty C_2^{1/2}}{C_1-C_1^{1/2}}\right)\frac{(C_1C_2)^{1/2}}{C_1-C_4},\\[8pt]
        A_3(g,b,\sigma,\eta)&=\frac{1}{C_1^{1/2}-C_4}\left(\frac{\|Dg_0\|_\infty C_2^{1/2}}{C_1^{1/2}-C_1}(C_1C_2)^{1/2}+\|Dg_0\|_\infty C_3^{1/2}\right)\,.
    \end{align*}
\end{corollary}
\begin{proof}
Set
\[
z_n^{(1)}:=\|Dg_n\|_\infty,\qquad
z_n^{(2)}:=\|D^2g_n\|_\infty,\qquad
z_n^{(3)}:=\|D^3g_n\|_\infty .
\]
By \Cref{lem_ug}, these quantities satisfy
\begin{align*}
    z_{n+1}^{(1)}&\le \sqrt{C_1}\,z_n^{(1)},\\
    z_{n+1}^{(2)}&\le C_1 z_n^{(2)}+\sqrt{C_2}\,z_n^{(1)},\\
    z_{n+1}^{(3)}&\le C_4 z_n^{(3)}+\sqrt{C_1C_2}\,z_n^{(2)}+\sqrt{C_3}\,z_n^{(1)}.
\end{align*}
Solving these recursions successively yields the stated bound.
\end{proof}


\subsection{Global Convergence Rate in Infinite Dimensions}\label{sec:global_proof}
With all pieces ready, we can assemble~\Cref{lem_one_step} and~\Cref{cor_gn3bound} to finalize the total error analysis. To do so, we first simplify our notations slightly:
\begin{align}\label{bar_A_def}
    |A_i(\eta)|\le\bar{A_i}(\eta),\quad\forall~i=1,2,3.
\end{align}
The explicit formula for $\bar{A_i}$ will be provided in \Cref{form_Ai}. In addition, we set simple notation:
\begin{align*}
    J_i:=\left(\|D^ib\|_\infty^2+\|D^i\sigma\|_\infty^2\right)^{1/2},\quad\forall~i=1,2,3.
\end{align*}
The final theorem we have is the following:


\begin{theorem}[Error Estimate]\label{thm_weak_conv}
    Suppose that $\phi_n$ and $\tilde\varphi_n$ are stochastic iterative processes starting from a fixed initial state $\phi_0\in H$ generated by the first and second stochastic iterations in \eqref{iter_disc_a}-\eqref{iter_disc_b}, respectively. Assume that $b$ and $\sigma$ have  bounded derivatives up to the third order and satisfy
    \begin{align*}
        \sup_{\phi\in H}\mathbb E\Big(\|\sigma(\phi)(\Delta W)\|^3+\|V(\phi)\|^3\Big)=:\mathcal D<\infty.
    \end{align*}
    Then, for each $\bar\eta\ge0$, there exists a constant $\mathcal C_{g,b,\sigma}(\bar\eta)$ such that
    \begin{align*}
        \max_{k=0,\cdots,\lfloor T/\eta\rfloor}\left|\mathbb E\big(g(\tilde\varphi_k)-g(\phi_k)\big)\right|\le\eta^2\mathcal D\mathcal C_{g,b,\sigma}(\bar\eta),\quad\forall~\eta\in[0,\bar\eta].
    \end{align*}
    The function $\bar\eta\mapsto\mathcal C_{g,b,\sigma}(\bar\eta)$ is defined as follows:
    \begin{align*}
        \mathcal C_{g,b,\sigma}(\bar\eta):=\left(\bar{A_1}(\bar\eta)\frac{\exp\left(T\left(4\|Db\|_\infty+6\bar\eta J_1^2\right)\right)-1}{4\|Db\|_\infty}+\bar{A_2}(\bar\eta)\frac{\exp\left(T\left(2\|Db\|_\infty+\bar\eta J_1^2\right)\right)-1}{2\|Db\|_\infty}\right.\\
        \left.+\bar{A_3}(\bar\eta)\frac{\exp\left(T\left(2\|Db\|_\infty+\bar\eta J_1^2\right)/2\right)-1}{\|Db\|_\infty}\right).
    \end{align*}
\end{theorem}

\begin{proof}
    By induction~\eqref{eqn:induction}, we have the definition of $G_l$ and we write them in a conditional form:
    \begin{equation}
        \begin{cases}
            G_k=\mathbb E\big(\mathbb E\big(g\big(\tilde\varphi_k(\phi_{k-1},k-1)\big)\big|\phi_{k-1}\big)\big)-\mathbb E\big(\mathbb E\big(g\big(\phi_k(\phi_{k-1},k-1)\big)\big|\phi_{k-1}\big)\big)\\
            G_l = \mathbb E\Big(\mathbb E\big(\mathbb E\big(g\big(\tilde\varphi_k(\tilde\varphi_l)-g\big(\tilde\varphi_k(\phi_l)\big)\big|\tilde\varphi_l(\phi_{l-1},l-1),\phi_l(\phi_{l-1},l-1)\big)\big|\phi_{l-1}\big)\Big)
        \end{cases}
    \end{equation}

We control these two terms separately.

    \noindent {\em Estimate of $G_k$.} Apply \Cref{lem_one_step}, we have
    \begin{align}\label{G_k_bound}
        \begin{aligned}
        G_k&=\mathbb E\big(\mathbb E\big(g\big(\tilde\varphi_k(\phi_{k-1},k-1)\big)\big|\phi_{k-1}\big)\big)-\mathbb E\big(\mathbb E\big(g\big(\phi_k(\phi_{k-1},k-1)\big)\big|\phi_{k-1}\big)\big)\\[8pt]
        &\le\mathbb E\big(\big|\mathbb E\big(g\big(\tilde\varphi_k(\phi_{k-1},k-1)\big)-g\big(\phi_k(\phi_{k-1},k-1)\big)\big|\phi_{k-1}\big)\big|\big)\\[8pt]
        &\le\eta^3\|D^3 g\|_\infty\mathbb E\big(\|\sigma(\phi_{k-1})(\Delta W)\|^3+\|V(\phi_{k-1})\|^3\big)\le\eta^3\|D^3g\|_\infty\mathcal D.
        \end{aligned}
    \end{align}

    \noindent {\em Estimate of $G_l$.} The value $G_l$ represents the comparison of the two states propagated to step $k$ starting from the slightly different states $\phi_l$ and $\tilde\varphi_l$ at step $l$, which are generated from the common previous state $\phi_{l-1}$ via the stochastic iterations in \eqref{iter_disc_a}--\eqref{iter_disc_b}.
    In this construction, by \Cref{lem_ug}, we can inductively generate new test functions $g_n$ for $n=1,\cdots,k-l$ via the operator $A$ defined in \eqref{oper_A}, thereby reducing the propagation and obtaining the one-step error representation. Under the conditional expectation with $\phi_{l-1}$ is fixed, we take $g_0=g$ and express the subsequent relation accordingly
    \begin{align*}
        &\mathbb E\Big(g_0\big(\tilde\varphi_k(\tilde\varphi_l)-g_0\big(\tilde\varphi_k(\phi_l)\big)\big|\tilde\varphi_l(\phi_{l-1},l-1),\phi_l(\phi_{l-1},l-1)\Big)\\[8pt]
        &\hspace{.5cm}=\mathbb E\Big(\mathbb E\big(g_0\big(\tilde\varphi_k(\tilde\varphi_{k-1}(\tilde\varphi_l,l))-g_0\big(\tilde\varphi_k(\tilde\varphi_{k-1}(\phi_l,l))\big)\big|\tilde\varphi_{k-1}(\tilde\varphi_l,l),\tilde\varphi_{k-1}(\phi_l,l)\big)\big|\tilde\varphi_l(\phi_{l-1},l-1),\phi_l(\phi_{l-1},l-1)\Big)\\[8pt]
        &\hspace{.5cm}=\mathbb E\Big(g_1\big(\tilde\varphi_{k-1}(\tilde\varphi_l,l)\big)-g_1\big(\tilde\varphi_{k-1}(\phi_l,l)\big)\big|\tilde\varphi_l(\phi_{l-1},l-1),\phi_l(\phi_{l-1},l-1)\Big)\\[8pt]
        &\hspace{.5cm}=\cdots=\mathbb E\Big(g_{k-l}\big(\tilde\varphi_l(\tilde\varphi_l,l)\big)-g_{k-l}\big(\tilde\varphi_l(\phi_l,l)\Big)\big|\tilde\varphi_l(\phi_{l-1},l-1),\phi_l(\phi_{l-1},l-1)\Big)\\[8pt]
        &\hspace{.5cm}=\mathbb E\Big(g_{k-l}(\tilde\varphi_l)-g_{k-l}(\phi_l)\big|\tilde\varphi_l(\phi_{l-1},l-1),\phi_l(\phi_{l-1},l-1)\Big).
    \end{align*}
    Then, we again use the one-step error estimate in \Cref{lem_one_step} to get
    \begin{align}\label{G_l_bound}
        G_l&\le\eta^3\|D^3g_{k-l}\|_\infty\mathbb E\Big(\mathbb E\big(\|\sigma(\phi_{l-1})(\Delta W)\|^3+\|V(\phi_{l-1})\|^3\big|\phi_{l-1}\big)\Big)\le\eta^3\|D^3g_{k-l}\|_\infty\mathcal D.
    \end{align}

    \noindent Finally, we combine \eqref{eqn:induction}, \eqref{G_k_bound} and \eqref{G_l_bound} to deduce
    \begin{align}\label{upp_bdd}
        \begin{aligned}
            \left|\mathbb E\,g(\tilde\varphi_k)-\mathbb E\,g(\phi_k)\right|&\le
            \eta^3\mathcal D
            \left(
            \|D^3 g_0\|_\infty+\sum_{l=1}^{k-1}\|D^3 g_{k-l}\|_\infty
            \right)=\eta^3\mathcal D\sum_{m=0}^{k-1}\|D^3g_m\|_\infty\\
            &\le\eta^3\mathcal D\left(A_1\frac{C_4^k-1}{C_4-1}+A_2\frac{C_1^k-1}{C_1-1}+A_3\frac{C_1^{k/2}-1}{C_1^{1/2}-1}\right),
        \end{aligned}
    \end{align}
    where we use \Cref{cor_gn3bound} in the last inequality. Note that this upper bound increases as $k$ grows. Therefore, it suffices to show that the upper bound is $O(\eta^2)$ when $k=\lfloor T/\eta\rfloor$.
    
    By the definition of $\bar{A_i}$ in \eqref{bar_A_def}, we may write
    \begin{align*}
        \mbox{RHS of \eqref{upp_bdd}}\le\eta^3\mathcal D\left(\bar{A_1}(\eta)\frac{C_4^{T/\eta}-1}{C_4-1}+\bar{A_2}(\eta)\frac{C_1^{T/\eta}-1}{C_1-1}+\bar{A_3}(\eta)\frac{C_1^{T/2\eta}-1}{C_1^{1/2}-1}\right).
    \end{align*}
    We use $1+x\le e^x$ and $C_4-1\ge 4\eta\|Db\|_\infty$ to get
    \begin{align*}
        \eta\frac{C_4^{T/\eta}-1}{C_4-1}\le\frac{\exp\left(T\left(4\|Db\|_\infty+6\bar\eta J_1^2\right)\right)-1}{4\|Db\|_\infty},\quad\forall~\eta\in[0,\bar\eta].
    \end{align*}
    By the similar way, one can derive the desired result, which ends the proof.
\end{proof}
\begin{remark}
    In \Cref{thm_weak_conv}, we assume that the derivatives up to 3rd of the drift $b$ and the diffusion $\sigma$ are uniformly bounded. This assumption is inherently tied to our discrete telescoping argument, as it prevents the degree of the test functions from exploding during the iterative compositions. However, thanks to the uniform moment bounds established in estimate \eqref{wp5}, it is theoretically plausible to relax this condition to allow for polynomial growth. Such an extension would likely require shifting from the current discrete framework to continuous-time stochastic analysis tools, such as the Kolmogorov Backward Equation, which we leave as a direction for future work.
\end{remark}

\section{Numerical Experiments}\label{sec:numerics}
In this section we present numerical experiments that validate the theoretical prediction in \Cref{thm_weak_conv}, which establishes an $O(\eta^2)$ weak convergence rate between the SGD iterates and the discretized stochastic modified equation (SME).

We begin in \Cref{subsec_homo} with a quadratic optimization problem subject to homogeneous noise. This simplified setting allows us to clearly compare three dynamics: the SGD iterates, the continuous-time SDE limit (SME), and its time-discretized counterpart. We then turn in \Cref{subsec_inhomo} to a nonlinear inverse problem with inhomogeneous (state-dependent) noise.

In practice, three sources contribute to the numerical error:
\begin{itemize}
\item \textbf{Finite-dimensional approximation error.} Functions in the Hilbert space $H$ require infinitely many coefficients for an exact representation and therefore cannot be simulated directly. In practice we represent functions using a truncated basis expansion. Let $H=\text{span}\{e_i\}_{i=1}^\infty$ denote the underlying basis and define the truncated space $H_D=\text{span}\{e_i\}_{i=1}^D$, where $D$ is the number of basis functions used in the simulation.

\item \textbf{Monte Carlo sampling error.}
The weak error \eqref{eq:weak_num} is defined in expectation. In numerical experiments each realization produces a different outcome. To approximate the expectation, we run both the SGD and the discrete-SME dynamics independently $N$ times and estimate the error by comparing the empirical means.

\item \textbf{Discretization error.}
According to \Cref{thm_weak_conv}, the weak discretization error scales as $O(\eta^2)$.

\item {\textbf{Choice of diffusion factorization in the numerical scheme.}
In \Cref{subsec_sme}, the diffusion operator $\sigma$ is constructed in the form $\sigma:L_\gamma^2\to H$ in order to ensure well-posedness of the SME. For the numerical implementation, however, we generate the Gaussian noise in the discrete-SME scheme through a square-root (e.g., Cholesky) factorization of the projected covariance operator. This does not affect the discrete-time approximation, since at the numerical level the distribution of the Gaussian increment, and hence the weak error behavior, is determined only by the covariance, not by the particular factorization used to realize it.
}
\end{itemize}

The primary objective of the experiments is to verify this discretization error. To isolate this effect, the influence of the truncation error and the Monte Carlo error must be kept sufficiently small. In practice, we choose $N$ large enough to reduce the sampling error, and we report convergence results for several values of the truncation dimension $D$. We claim the second order discretization error is captured if the convergence rate is observed across all values of $D$ before saturating (when MC error dominates). Details of these choices are provided in the two examples below.

\subsection{Example 1: Homogeneous Noise (Quadratic Example)}\label{subsec_homo}

We consider a model problem where the Hilbert space ${H}$ is endowed with an orthonormal basis $\{e_{i}\}_{i=1}^{\infty}$ and the objective function $\mathcal{F}$ is quadratic. Specifically, let
\begin{align*}
    \mathcal{F}(\phi) := \mathbb{E}_{\gamma}[\mathcal{F}_\gamma(\phi)], \quad \text{with} \quad \mathcal{F}_\gamma(\phi) := \langle \phi - \gamma, \Lambda(\phi - \gamma) \rangle_{{H}}, \quad \forall \gamma \in {H},
\end{align*}
where $\Lambda \in \mathcal{L}^+({H})$. Assuming the random variable $\gamma$ satisfies $\mathbb{E}[\gamma] = 0$ and the covariance operator $\mathbb{E}[\gamma \otimes \gamma] = Q$, the objective function can be written as $\mathcal{F}(\phi) = \langle \phi, \Lambda \phi \rangle + \Tr(\Lambda Q)$.

The SGD iteration is given by
\begin{align}\label{sgd_homo}
\phi_{n+1} = \phi_n - \eta \nabla \mathcal{F}(\phi_n) + \eta \bigl( \nabla \mathcal{F}(\phi_n) - \nabla \mathcal{F}_{\gamma_n}(\phi_n) \bigr),
\end{align}
where $\{\gamma_n\}$ are i.i.d. samples. Due to the quadratic structure of $\mathcal{F}_\gamma$, the stochastic noise term simplifies to
\[
\nabla \mathcal{F}(\phi_n) - \nabla \mathcal{F}_{\gamma_n}(\phi_n) = -2\Lambda \gamma_n.
\]

Notably, this noise term is independent of the state $\phi_n$. Consequently, the corresponding SME takes on the form of \eqref{sde_model} with a constant Hilbert-Schmidt diffusion operator $\sigma$ satisfying
\begin{align*}
    \sigma \sigma^* = \text{Var}(2\Lambda \gamma) = 4 \Lambda Q \Lambda^*.
\end{align*}
The discretized version of the SME is
\begin{align}\label{disc_sme_homo}
    \tilde{\varphi}_{n+1} = \tilde{\varphi}_n - \eta \nabla \mathcal{F}(\tilde{\varphi}_{n}) + \eta \sigma \Delta W_{n},
\end{align}
where $\Delta W_n$ are i.i.d. increments of a $Q$-Wiener process.

In our numerical experiments, we use the following configurations:
\begin{align*}
\langle e_i, \Lambda e_j \rangle = (0.8)^{|i-j|}, \quad \langle e_i, Q e_j \rangle = \frac{\delta_{ij}}{i^2}, \quad \text{and} \quad \gamma = \sum_{i=1}^\infty \frac{\zeta_i}{i} e_i,
\end{align*}
where $\zeta_i$ are Bernoulli-type random variables satisfying $\mathbb{P}(\zeta_i = -0.5) = 0.8$ and $\mathbb{P}(\zeta_i = 2) = 0.2$. For the simulation, we set $\phi_0 = 0$ with a terminal time $T = 1.0$, and we measure the weak error via the test function $g(\phi) = \|\phi\|^4$:
\begin{equation}\label{eq:weak_num}
\mathcal{E}_{weak} = \left| \mathbb{E}[g(\phi_{T/\eta})] - \mathbb{E}[g(\tilde{\varphi}_{T/\eta})] \right|.
\end{equation}

Since the objective function is quadratic, the SME admits an analytical solution. The standard Euler--Maruyama scheme predicts a discretization error of order $\mathcal{O}(\eta)$, which we verify numerically in \Cref{fig_homo_analy}. To observe the statistical behavior of the SGD solver, we run $N=5 \times 10^6$ independent trajectories and plot the average.
\begin{figure}[h!]
    \centering
    \begin{overpic}[width=0.6\linewidth]{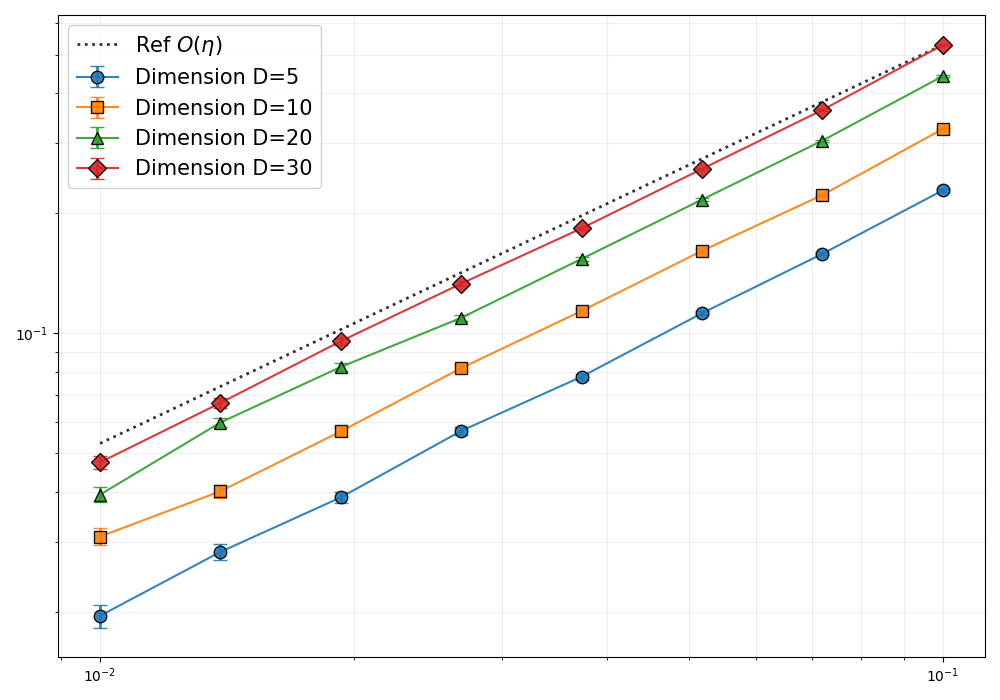}
        \put(44,0){stepsize $\eta$}
        \put(-6,12){\rotatebox{90}{$\Big|\frac{1}{N}\sum_{l=1}^N[g(\phi_{T/\eta}^{(l)})-\mathbb Eg(\varphi_T)\Big|$}}
    \end{overpic}
    \caption{{\textbf{(Example 1, MC SGD vs. SME)} Weak error between the SGD iterate and the exact SME expectation estimated using $N=5\times10^6$ Monte Carlo samples. Results are shown for different truncation dimensions $D$ of the Hilbert space. The dotted line indicates the reference slope $\mathcal{O}(\eta)$.}}
    \label{fig_homo_analy}
\end{figure}

To validate the theory and visualize the convergence rate in $\eta$, we ensure that the impact of the Monte Carlo sampling error is negligible. We examine the number of samples $N$ required for the error to reach a plateau. As shown in \Cref{fig_trials}, for all fixed $\eta$, as $N$ increases, the error initially decays at the standard Monte Carlo rate of $1/\sqrt{N}$ before entering a plateau regime where the discretization error dominates. Based on these results, $N = 5 \times 10^6$ is a sufficiently large sample size, and we employ this value for all subsequent experiments.
\begin{figure}[ht]
    \centering
    \begin{overpic}[width=0.6\linewidth]
    {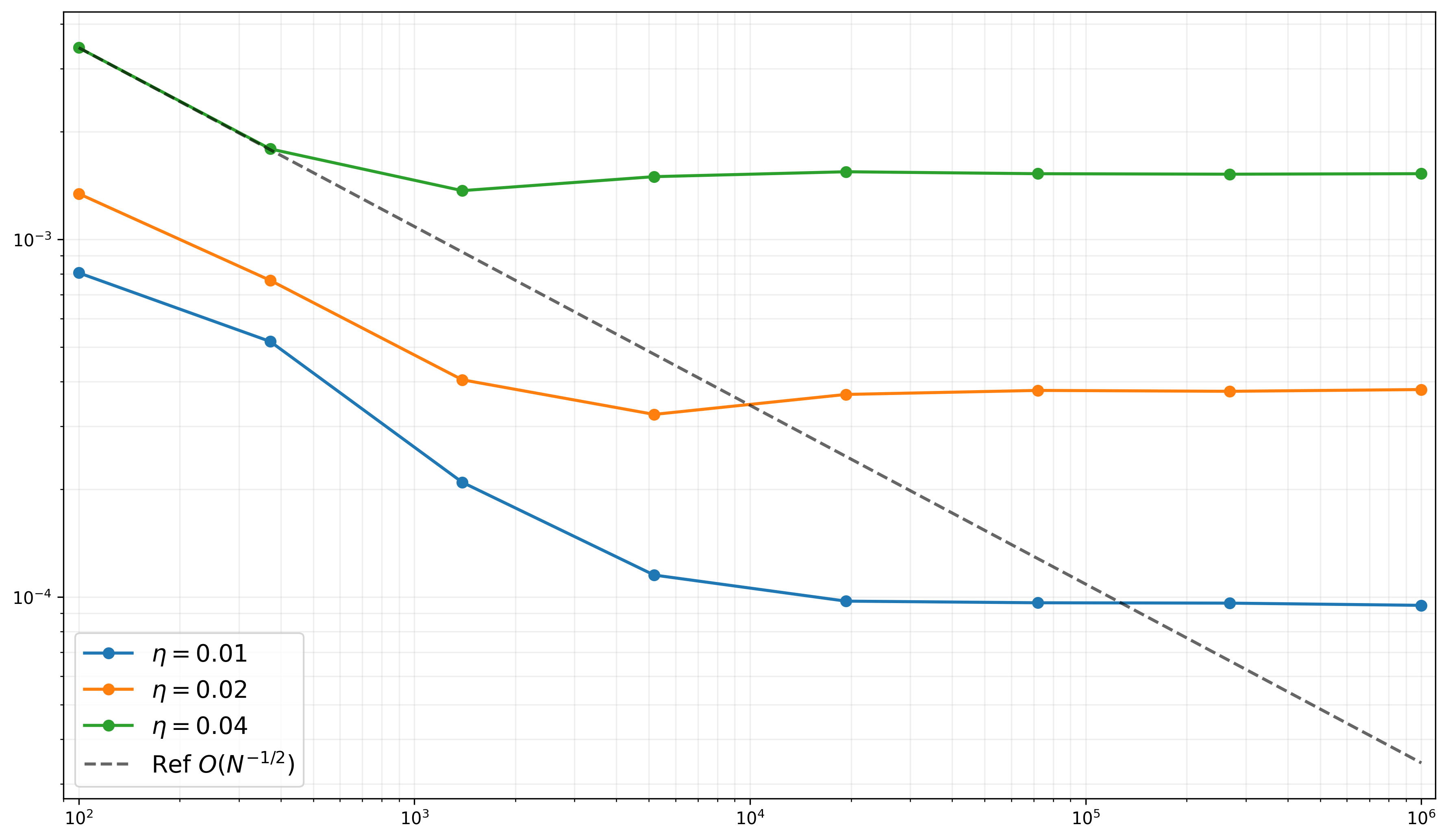}
        \put(27,-3){number of Monte Carlo trials $N$}
        \put(-6,7){\rotatebox{90}{$\Big|\frac{1}{N}\sum_{l=1}^Ng(\tilde\varphi_{T/\eta}^{(l)})-\mathbb Eg(\varphi_T)\Big|$}}
    \end{overpic}
    \caption{{\textbf{(Example 1, MC discretized SME vs. SME)} Convergence of the Monte Carlo estimator with respect to the number of trials $N$ for different stepsizes $\eta$, averaged over $50$ independent runs. For small $N$, the error decreases at the Monte Carlo rate $\mathcal{O}(N^{-1/2})$ till the error saturates and plateaus at a place where discretization error dominates for large $N$. The dashed line indicates the reference slope $\mathcal{O}(N^{-1/2})$. The experiment is performed with truncation dimension $D=30$.}}
    \label{fig_trials}
\end{figure}

In \Cref{fig_homo}, we examine the difference between the SGD iterates \eqref{sgd_homo} and the discretized SME \eqref{disc_sme_homo}. \Cref{fig_homo_log} uses a log-log scale, where the straight reference line illustrates the algebraic decay rate $\mathcal{O}(\eta^2)$ of the weak error, confirming the prediction in \Cref{thm_weak_conv}. In contrast, \Cref{fig_homo_regular} adopts a log-linear scale to highlight the monotonic convergence as $\eta$ decreases. We also include error bars corresponding to the standard deviation of the Monte Carlo estimator; as $\eta$ decreases, these bars shrink, indicating reduced variability. In both plots, it is clear that the convergence rate is second-order across all choices of $D$.
\begin{figure}[h!]
    \centering
    \hspace{.5cm}
    \begin{subfigure}[b]{0.45\textwidth}
        \centering
        \begin{overpic}
            [width=\linewidth]
            {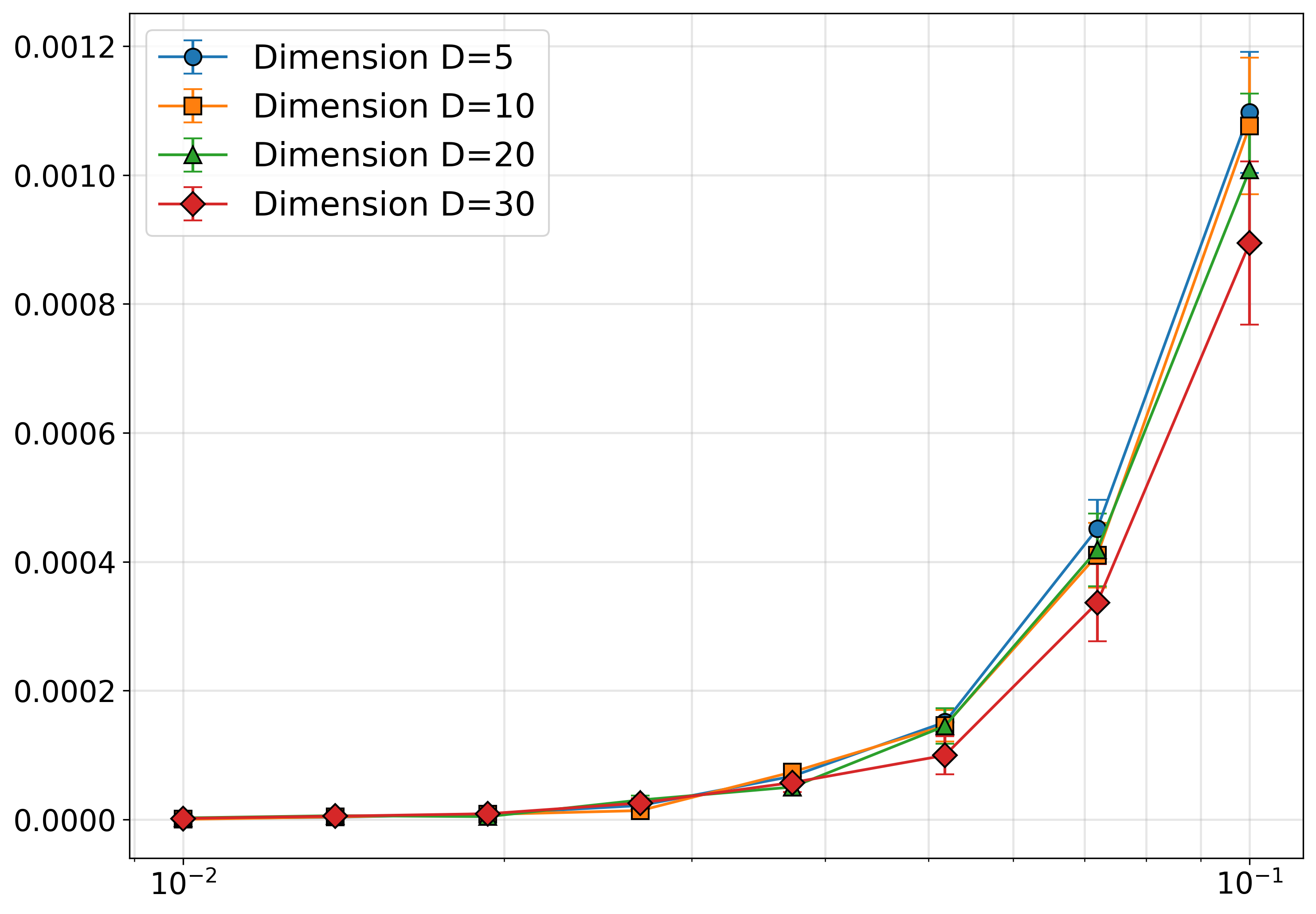}
            \put(43,-1){\footnotesize stepsize $\eta$}
            \put(-8,6){\rotatebox{90}{\footnotesize$\Big|\frac{1}{N}\sum_{l=1}^N\left(g(\phi_{T/\eta}^{(l)})-g(\tilde\varphi_{T/\eta}^{(l)})\right)\Big|$}}
        \end{overpic}
        \caption{log($x$)-regular($y$) scale}
        \label{fig_homo_regular}
    \end{subfigure}
    \hspace{.5cm}
    \begin{subfigure}[b]{0.45\textwidth}
        \centering
        \begin{overpic}
            [width=\linewidth]
            {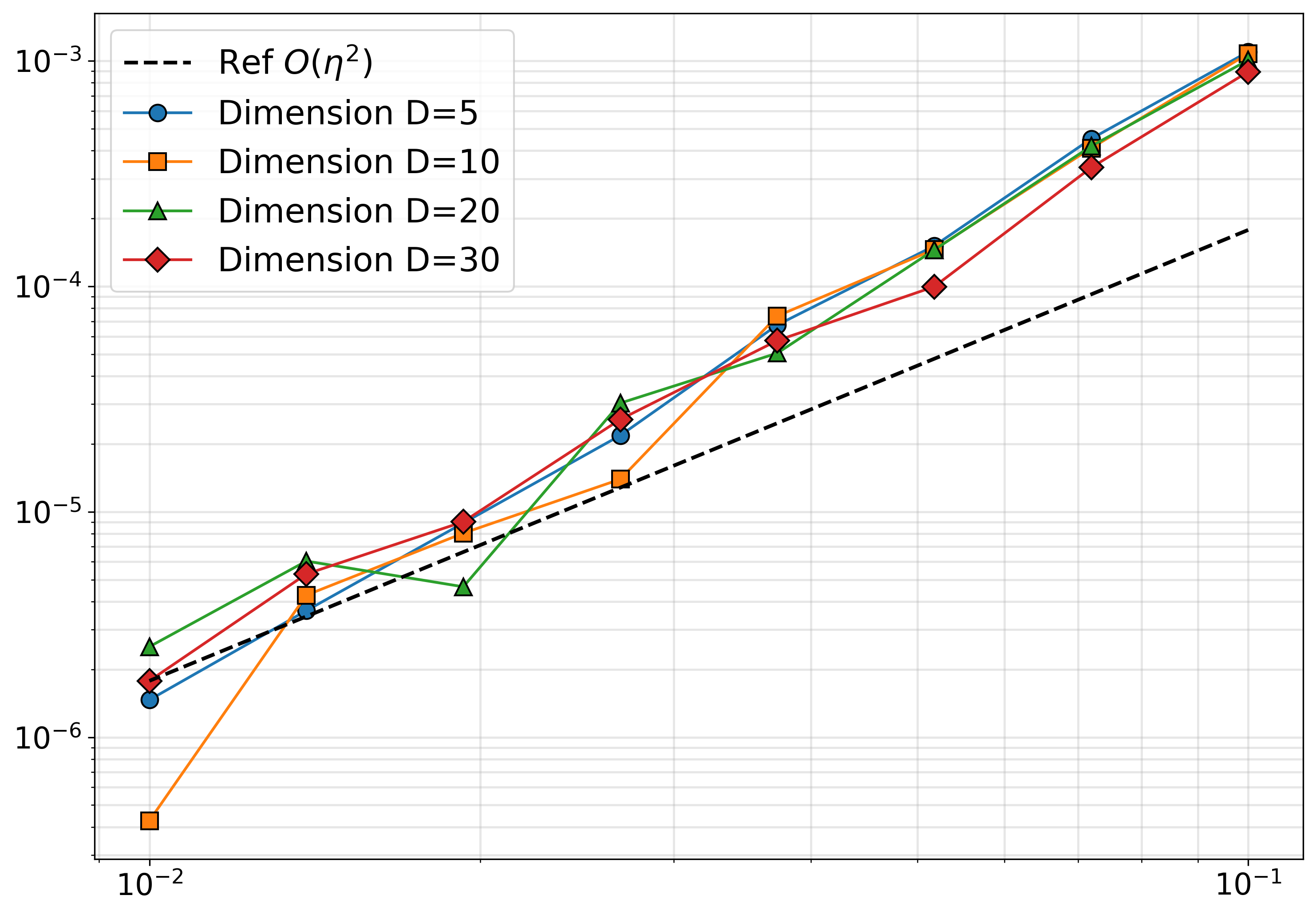}
        \end{overpic}
        \caption{log($x$)-log($y$) scale}
        \label{fig_homo_log}
    \end{subfigure}
    \caption{{\textbf{(Example 1, Weak error in the homogeneous noise Setting)} The plots show the weak error \eqref{eq:weak_num} as a function of the step size $\eta$ for truncation dimensions $D\in\{5,10,20,30\}$. Error bars in the left panel represent the Monte Carlo standard error (MCSE) computed from $5\times10^6$ samples. The dashed line indicates the reference slope $\mathcal O(\eta^2)$.}}
    \label{fig_homo}
\end{figure}

\subsection{Example 2: Inverse Problem}\label{subsec_inhomo}

In this second example, we investigate a setting where the stochastic gradient noise depends explicitly on the current state, leading to a state-dependent diffusion coefficient in the corresponding SME. We consider an inverse problem for an unknown target function $\bar\phi \in H = L^2([0,1]^2)$. The objective functional is defined as
\begin{align*}
    \mathcal{F}(\phi) := \mathbb{E}_{x \sim \mathcal{U}([0,1]^2)}[\mathcal{F}_x(\phi)], \quad \text{with} \quad \mathcal{F}_x(\phi) = \frac{1}{2}\left(\mathcal{A}_x(\phi) - \mathcal{A}_x(\bar\phi)\right)^2,
\end{align*}
where $\mathcal{A}_x$ is a sensing operator parameterized by $x$. This problem originates from the optical ptychography setting~\cite{pfeiffer2018x}. We define $\mathcal{A}_x$ via a Gaussian kernel centered at $x$:
\begin{align*}
    \mathcal{A}_x(\phi) := \int_{[0,1]^2} \mathcal{K}_\varepsilon(x,y)\,\phi(y)\,dy, \quad \text{with} \quad \mathcal{K}_\varepsilon(x,y) = \frac{1}{2\pi\varepsilon^2}\exp\!\left(-\frac{\|x-y\|^2}{2\varepsilon^2}\right),
\end{align*}
where $\varepsilon > 0$ is a fixed small parameter representing the resolution of the measurement at location $x$.

The gradients $\nabla\mathcal{F}$ and $\nabla\mathcal{F}_x$ can be computed explicitly using the linearity of the sensing operator and the spectral representation of the Gaussian kernel. We expand $H$ using the standard sinusoidal basis $\{e_k\}$:
\begin{align}\label{sine_basis}
    e_k(x_1,x_2) = 2\sin(k_1\pi x_1)\sin(k_2\pi x_2), \quad k=(k_1,k_2)\in\mathbb{N}^2,
\end{align}
with associated eigenvalues $\lambda_k = \pi^2\|k\|^2$ for the Laplacian. We denote the coefficients of the expansion as:
\begin{align*}
    c_k(\phi) = \langle \phi, e_k \rangle \quad \text{and} \quad \bar{c}_k = \langle \bar\phi, e_k \rangle, \quad \forall k \in \mathbb{N}^2.
\end{align*}
Consequently, the stochastic gradient is given by:
\begin{align*}
    \nabla\mathcal{F}_x(\phi) &= \mathcal{K}_\varepsilon(x,\cdot) \int_{[0,1]^2} \mathcal{K}_\varepsilon(x,y)(\phi(y)-\bar\phi(y))\,dy \\
    &= \sum_{k \in \mathbb{N}^2} \exp\left(-\frac{\varepsilon^2\lambda_k}{2}\right) (c_k(\phi)-\bar c_k) e_k(x) \sum_{l \in \mathbb{N}^2} \exp\left(-\frac{\varepsilon^2\lambda_l}{2}\right) e_l(x) e_l.
\end{align*}
Taking the expectation over $x \in [0,1]^2$ yields:
\[
    \nabla\mathcal{F}(\phi) = \sum_{k \in \mathbb{N}^2} (c_k(\phi)-\bar c_k) \exp(-\varepsilon^2\lambda_k) e_k.
\]
The noise term in the basis representation is then:
\[
    \langle \nabla\mathcal{F}(\phi) - \nabla\mathcal{F}_x(\phi), e_k \rangle = (c_k(\phi)-\bar c_k)e^{-\varepsilon^2\lambda_k} - \sum_{l \in \mathbb{N}^2} e^{-\frac{\varepsilon^2(\lambda_l+\lambda_k)}{2}} (c_l(\phi)-\bar c_l) e_l(x) e_k(x).
\]
From this, the entries of the state-dependent covariance operator $Q(\phi)$ can be computed as:
\begin{align*}
    &\mathbb{E} \left[ \langle \nabla(\mathcal{F}-\mathcal{F}_x)(\phi), e_k \rangle \langle \nabla(\mathcal{F}-\mathcal{F}_x)(\phi), e_{k'} \rangle \right] \\
    &= (c_k(\phi)-\bar c_k)(c_{k'}(\phi)-\bar c_{k'}) e^{-\varepsilon^2(\lambda_k+\lambda_{k'})} \\
    &\quad - \sum_{l,l'} e^{-\frac{\varepsilon^2(\lambda_l+\lambda_{l'}+\lambda_k+\lambda_{k'})}{2}} (c_l(\phi)-\bar c_l)(c_{l'}(\phi)-\bar c_{l'}) \mathbb{E} \left[ e_l(x)e_k(x)e_{l'}(x)e_{k'}(x) \right].
\end{align*}
This expression defines the covariance matrix used to formulate the diffusion term in the SME.

{In the numerical implementation, the expectation with respect to 
$x\sim\mathcal U([0,1]^2)$ is approximated using a uniform grid 
$G_x=\{x_m\}_{m=1}^{N_x\times N_x}$ on $[0,1]^2$. At each SGD step, the location $x$
is also sampled from the grid points in $G_x$. In the simulations, $N_x=30$ is used for \Cref{fig_target}, \Cref{fig_grad} and \Cref{fig_grad_pt}, and $N_x=10$ for \Cref{fig_inhomo}.

To construct the covariance operator $Q(\phi)$, we evaluate the basis 
functions $\{e_k(x_m)\}$ on the grid and compute empirical averages of the 
corresponding products over $G_x$. In particular, all expectations appearing 
in the covariance expression are obtained numerically through grid averaging, 
rather than by evaluating the integrals analytically.

}

For the numerical simulation, we choose a target function $\bar\phi$ given by:
\begin{align*}
    \bar\phi(x_1,x_2) = x_1(1-x_1)x_2(1-x_2) \sum_{n=1}^3 a_n \exp\left(-\beta_n \left((x_1-c_{n,1})^2 + (x_2-c_{n,2})^2\right)\right),
\end{align*}
with parameters:
\begin{gather*}
    (a_1,a_2,a_3)=(1,0.8,0.65), \quad (\beta_1,\beta_2,\beta_3)=(35,30,28), \\
    c_1=(0.2,0.8), \quad c_2=(0.6,0.4), \quad c_3=(0.3,0.2).
\end{gather*}
{The reference solution is visualized in \Cref{fig_target},  which exhibits three localized peaks with smooth decay towards the boundary.}
\begin{figure}[htb]
    \centering
    \begin{overpic}[width=0.45\linewidth]{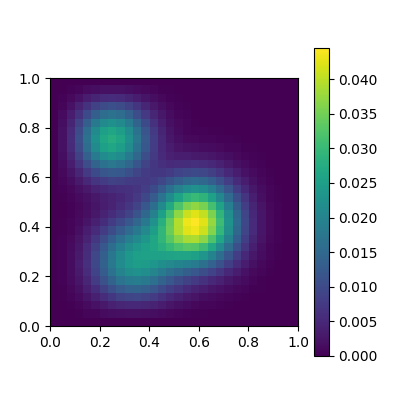}
        \put(35,82.5){$K=20$}
    \end{overpic}
    \caption{\textbf{(Example 2: Target function)} Heatmap of the target function $\bar\phi$ used in the inverse problem, visualized using a truncation level $K=20$.}
    \label{fig_target}
\end{figure}

We initialize at $\phi_0 = 0$ and use the test function $g(\phi) = \|\phi\|^4$ to evaluate the weak error. To demonstrate that the convergence is uniform with respect to discretization, we truncate $H$ using the first $K \times K$ modes, for $K \in \{2, 3, 4, 6, 8\}$. {To illustrate the effect of the truncation, \Cref{fig_grad} shows the gradient $\nabla\mathcal F(0)$ reconstructed using different values of $K$. As the truncation level increases, the representation becomes more localized and better resolves the underlying structure of the objective functional.}
\begin{figure}[htb]
    \centering
    \begin{overpic}[width=0.9\linewidth]{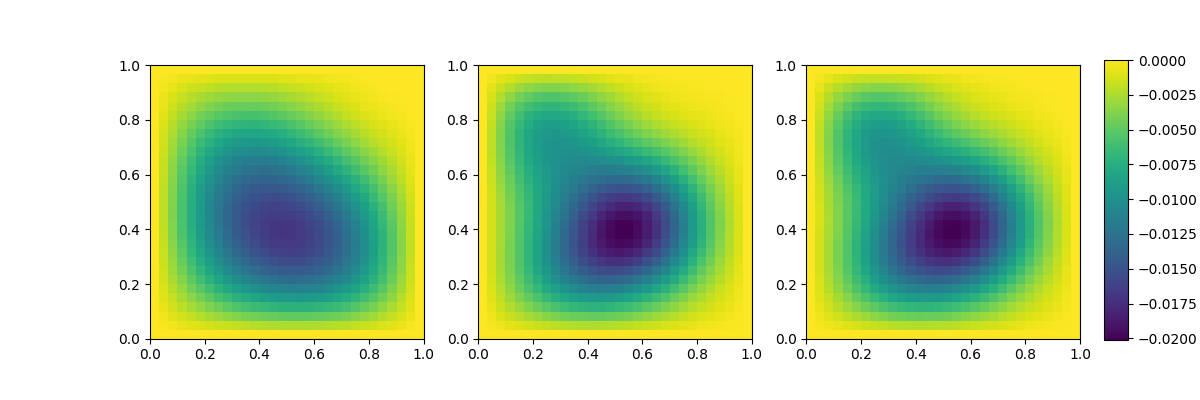}
        \put(21,29){$K=2$}
        \put(48,29){$K=4$}
        \put(75,29){$K=8$}
    \end{overpic}
    \caption{\textbf{(Example 2: Gradient of the objective functional)} Heatmap of $\nabla\mathcal F(\phi)$ evaluated at $\phi=0$, reconstructed using different truncation levels $K$.}
    \label{fig_grad}
\end{figure}
{In \Cref{fig_grad_pt}, we further plot an example of the stochastic gradient $\nabla\mathcal F_x(0)$. Compared with the full gradient $\nabla\mathcal F(0)$ (shown in \Cref{fig_grad}), the stochastic gradient exhibits a strong spatial concentration around the measurement location.}
\begin{figure}[htb]
    \centering
    \begin{overpic}[width=0.9\linewidth]{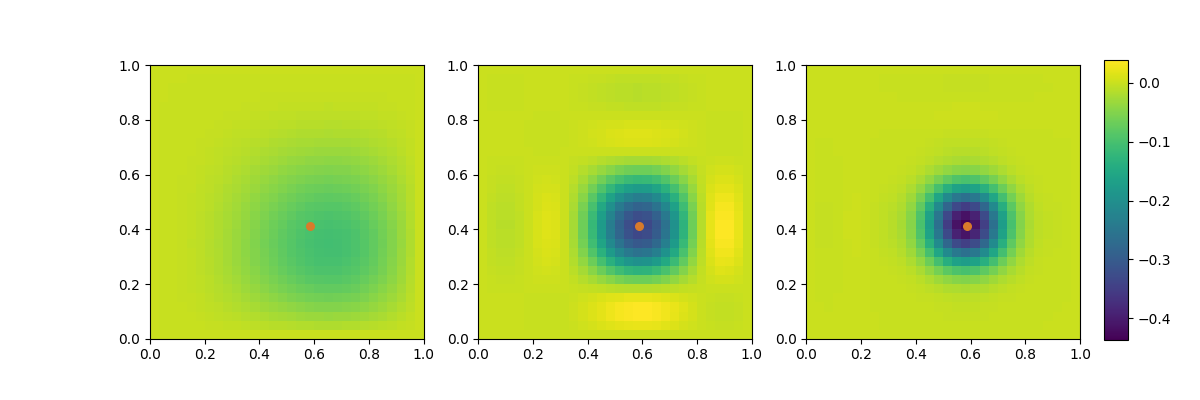}
        \put(21,29){$K=2$}
        \put(48,29){$K=4$}
        \put(75,29){$K=8$}
    \end{overpic}
    \caption{\textbf{(Example 2: Stochastic gradient)} Heatmap of the stochastic gradient $\nabla\mathcal F_x(\phi)$ evaluated at $\phi=0$ for a sampled location $x$ (orange dot). The plots show the reconstruction using different truncation levels $K$.}
    \label{fig_grad_pt}
\end{figure}

The resulting weak errors are presented in \Cref{fig_inhomo}. The error decreases with a clear slope of $\mathcal{O}(\eta^2)$ across all truncation dimensions $K$, which is in excellent agreement with the theoretical prediction of \Cref{thm_weak_conv}. Notably, as $K$ increases, the error curves remain consistent and do not show any degradation in the convergence rate. This suggests that the $\mathcal{O}(\eta^2)$ scaling is robust to the dimension of the subspace and is expected to persist in the infinite-dimensional limit.

\begin{figure}
    \centering
        \begin{overpic}[width=0.6\linewidth]{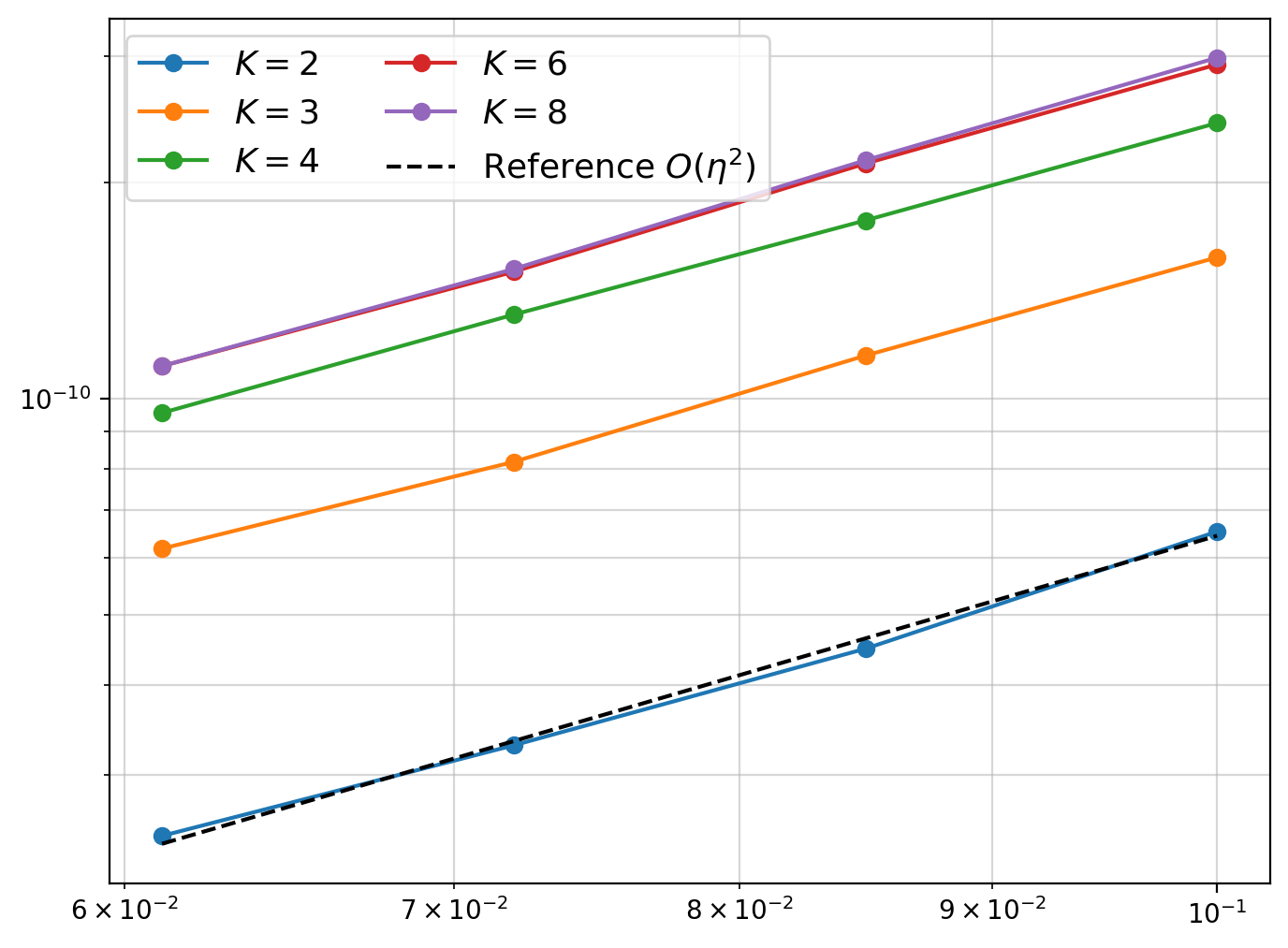}
            \put(44,-2){ stepsize $\eta$}
            \put(-7,12){\rotatebox{90}{$\Big|\frac{1}{N}\sum_{l=1}^N\left(g(\phi_{T/\eta}^{(l)})-g(\tilde\varphi_{T/\eta}^{(l)})\right)\Big|$}}
        \end{overpic}
    \caption{{\textbf{(Example 2: Weak error in the inhomogeneous noise setting)} The plot shows the difference between the SGD iterate and the discretized SME as a function of the step size $\eta$ for different truncation levels $K$. The dashed line indicates the reference slope $\mathcal{O}(\eta^2)$. All results are obtained using $10^6$ Monte Carlo trials {with smoothing parameter $\varepsilon=0.1$.}}}
    \label{fig_inhomo}
\end{figure}
{\subsection{Example 3: Restoration of the Cameraman Image}}
In this subsection, we keep the inverse-problem setting of \Cref{subsec_inhomo} and replace only the smooth synthetic target by the standard cameraman test image. The test image \(I^{\mathrm{cam}} \in [0,1]^{512\times 512}\) is a grayscale image that is normalized to take values in \([0,1]\). Here we use the standard grayscale convention: black corresponds to 0 and white to 1. It is a finite-dimensional object, and this subsection is devised completely for exploration purpose.

Since the groundtruth image is not an object in an Hilbert space, we are to process the data and map it to one. To do so, we first resample this array onto an \(N_x\times N_x\) computational grid by Lanczos interpolation, and interpret the resulting array \(I^{(N_x)} \in [0,1]^{N_x\times N_x}\) as grid values of the target and project it onto the first \(K\times K\) modes of the sine basis \eqref{sine_basis}. The produced function thus becomes
\[
u^\dagger_{N_x,D}(x,y)
=
\sum_{k_1=1}^{K}\sum_{k_2=1}^{K}
c_{k_1,k_2}\,e_{k_1,k_2}(x,y),
\]
where the coefficients \(c_{k_1,k_2}\) are obtained from the discrete \(L^2\)-projection of the resampled image \(I^{(N_x)}\) onto the span of the first \(K\times K\) sine modes, and \(D=K^2\) is the number of basis functions. This projected target is illustrated in \Cref{fig_cameraman_target}. Note that this final target is not identical to the raw resampled image due to the artificial effect of the basis functions.

\begin{figure}
    \centering
    \begin{overpic}[width=\linewidth]{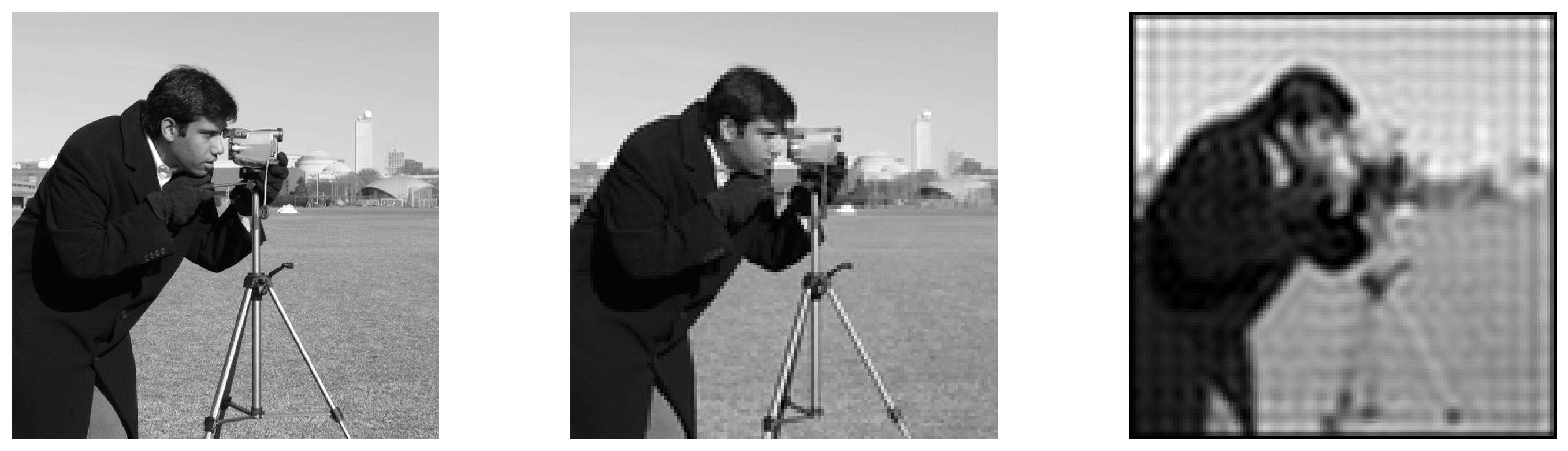}
        \put(7.5,29){Original Image}
        \put(41.5,29){Resampled Image}
        \put(77.5,29){Projected Target}
        \end{overpic}
    \caption{From left to right: the standard cameraman test image on a \(512\times512\) grid, its Lanczos-resampled version on a \(128\times128\) grid, and the corresponding sine-basis approximation with \(D=40^2\).}
    \label{fig_cameraman_target}
\end{figure}

We first report the quantitative weak-error result. We use the same numerical setting as in \Cref{fig_inhomo}, except that the smooth synthetic target is replaced by the cameraman target described above. 
To show the error plot, we use a lower-resolution target with \(N_x=10\) and \(D=8^2\), together with \(N=10^6\) Monte Carlo samples. The result is shown in \Cref{fig_cameraman_decayrate}. 
The same \(O(\eta^2)\) decay is clearly visible, across all values of $D=K^2$, indicating that the second-order weak matching between SGD and the discretized SME is still observed when the smooth synthetic target is replaced by a standard benchmark image with sharp edges.

\begin{figure}
    \centering
    \begin{overpic}[width=0.6\linewidth]{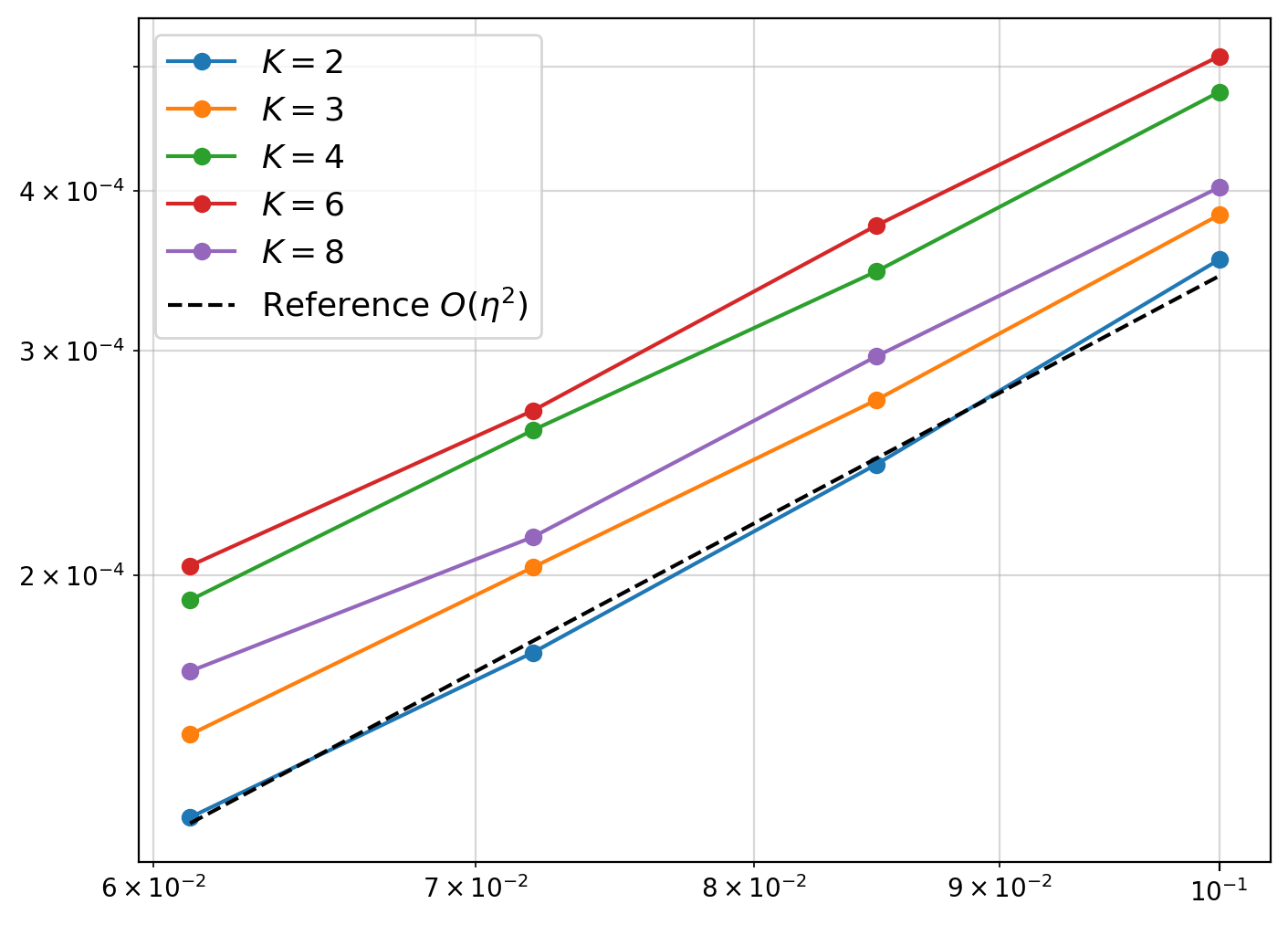}
        \put(44,-2){stepsize \(\eta\)}
        \put(-7,12){\rotatebox{90}{$\Big|\frac{1}{N}\sum_{l=1}^N\bigl(g(\varphi_{T/\eta}^{(l)})-g(\tilde\phi_{T/\eta}^{(l)})\bigr)\Big|$}}
    \end{overpic}
    \caption{Weak error between SGD and the discretized SME for the cameraman target. The numerical setting is the same as in \Cref{fig_inhomo}, except that the smooth synthetic target is replaced by the cameraman image. The dashed line indicates the reference slope \(O(\eta^2)\).}
    \label{fig_cameraman_decayrate}
\end{figure}

We also run the optimization over a higher-resolution target with \(N_x=128\), \(D=40^2\), \(\varepsilon=0.01\), and \(\eta=10^{-3}\), starting from the zero initial condition. The results are presented in~\Cref{fig_cameraman_sample_path}. In the upper row, we plot the reconstruction generated by SGD, and the lower row shows the one generated by the discretized SME, at times \(t=1,3,6,15\). 
Although this figure is not intended as a pathwise validation of the theory, it shows that both dynamics recover the same large-scale features of the image over time.

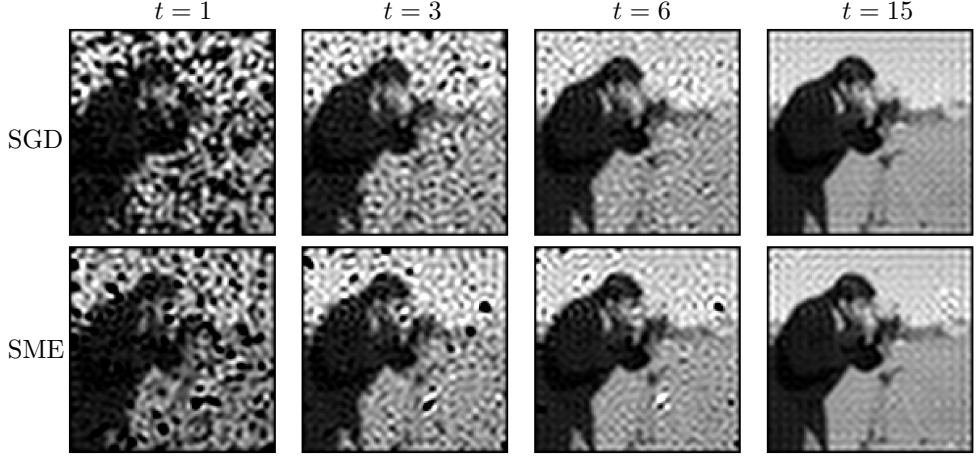
\begin{figure}[htb]
    \centering
    \begin{overpic}[width=0.8\linewidth]{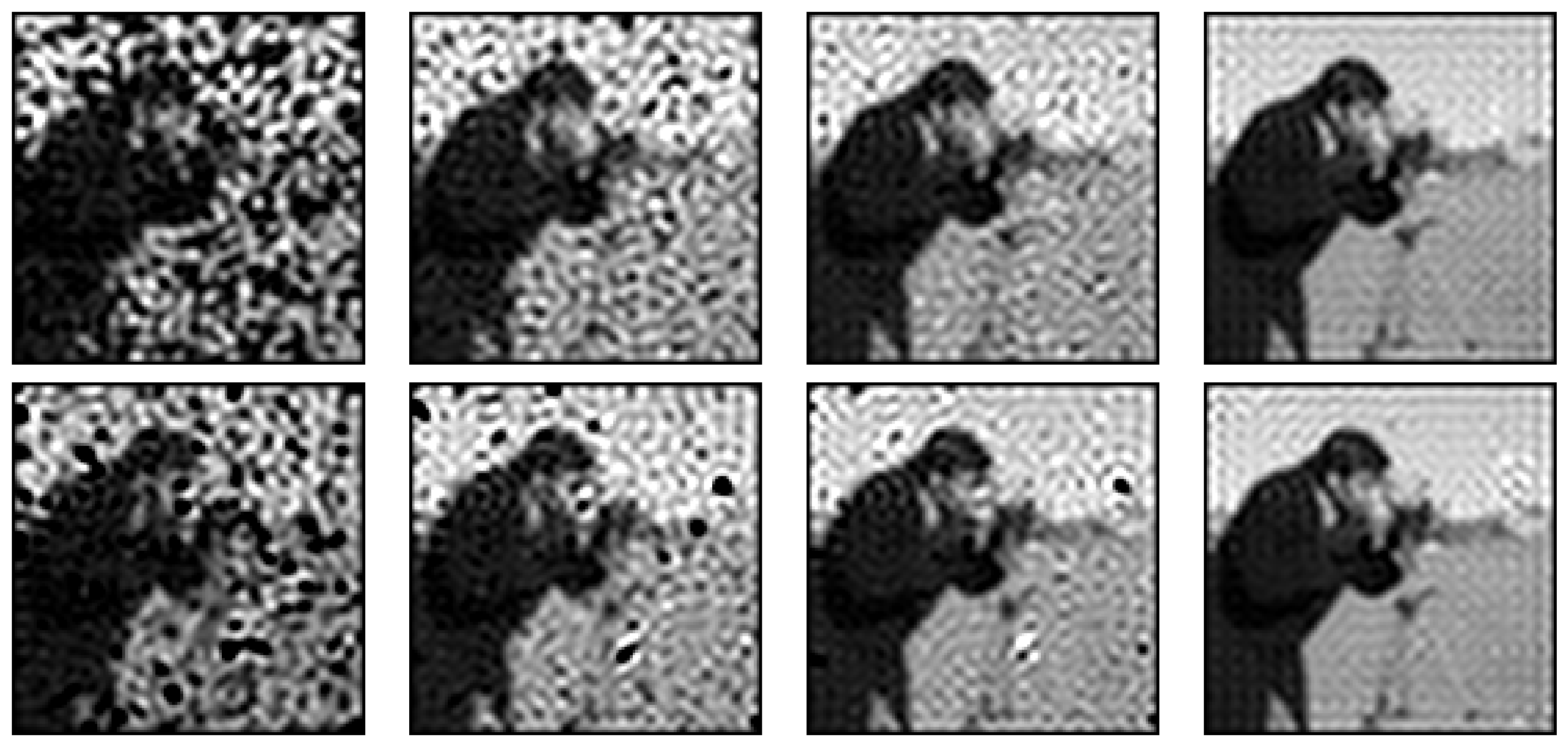}
        \put(-6,34){SGD}
        \put(-6,11){SME}
        \put(10,48){\(t=1\)}
        \put(35,48){\(t=3\)}
        \put(60,48){\(t=6\)}
        \put(85,48){\(t=15\)}
    \end{overpic}
    \caption{A single sample-path illustration of the reconstruction dynamics for the cameraman target. The upper row corresponds to SGD and the lower row to the discretized SME.}
    \label{fig_cameraman_sample_path}
\end{figure}

\section{Conclusion}
In this paper, we derived a stochastic modified equation (SME) for stochastic gradient descent (SGD) in infinite-dimensional Hilbert spaces. We successfully identified the appropriate space for the cylindrical Brownian motion and characterized the diffusion coefficients necessary to establish the system's well-posedness. A key finding of this work is the refined error analysis between the SGD iterates and the SME. While classical results typically suggest that the Euler-Maruyama discretization yields first-order weak accuracy, we demonstrate that the weak discrepancy between the SGD iterates and the continuous SME scales as ${O}(\eta^2)$. This result significantly extends existing weak approximation theories from finite-dimensional settings to the infinite-dimensional regime. Our theoretical framework was further validated through numerical experiments across both homogeneous and inhomogeneous noise regimes, confirming the predicted scaling laws.

Several avenues for future research remain open. One natural extension involves the study of the long-time behavior of the SME, specifically regarding its stability properties and the characterization of invariant measures as they relate to the asymptotic behavior of SGD. It would also be of interest to extend the present analysis to more general stochastic optimization schemes, such as momentum-based or adaptive methods, where the infinite-dimensional setting introduces unique analytical challenges. From a theoretical perspective, broadening the applicability of this framework may require relaxing the regularity assumptions on the drift and diffusion operators or allowing for weaker moment conditions on the stochastic gradient noise. Finally, a quantitative analysis of discretization and truncation effects remains essential to clarify the complex interaction between infinite-dimensional dynamics and their finite-dimensional numerical approximations.

\appendix
\section{Proof of \Cref{thm_wellposed}}\label{pf_thm_wellposed}
For every $n\in\mathbb N$, we define the truncation operator
\[
\pi_n(x)=
\begin{cases}
x,& \|x\|\le n,\\[8pt]
n\,\dfrac{x}{\|x\|},& \|x\|>n,
\end{cases}
\]
and we set
\[
b_n(x)=b(\pi_n(x)), \qquad
\sigma_n(x)=\sigma(\pi_n(x)).
\]
It is immediate to check that $b_n$ and $\sigma_n$ satisfy \eqref{wp1}, for some constant independent of $n$, namely
\begin{equation}
    \label{wp2}
    \langle b_n(x),x\rangle + \|\sigma_n(x)\|_{\mathcal L_2}^2
\le c(1+\|x\|^2).
\end{equation}
Moreover, $b_n$ and $\sigma_n$ are globally Lipschitz continuous, so that  the equation
\[
dX_t^n=b_n(X_t^n)\,dt+\sigma_n(X_t^n)\,dW_t,
\qquad X_0^n=x,
\]
admits a unique global adapted solution $X^n_t \in\,L^p(\Omega;C([0,T])$, for every $p\geq 1$ and $T>0$.

Now, we fix $p\ge1$ and apply It\^o's formula to $X_t^n$ and $\Phi(u)=\|u\|^{2p}$.
Since
\[
D\Phi(u)h=2p\|u\|^{2p-2}\langle u,h\rangle,
\]
and
\[
D^2\Phi(u)(h,k)
=2p(2p-2)\|u\|^{2p-4}\langle u,h\rangle\langle u,k\rangle
+2p\|u\|^{2p-2}\langle h,k\rangle.
\]
we get
\[
\begin{aligned}
\|X^n_t\|^{2p}
&= \|x\|^{2p}
+2p\int_0^t \|X^n_s\|^{2p-2}\langle X^n_s,b_n(X^n_s)\rangle\, ds
+p\int_0^t \|X^n_s\|^{2p-2}
\|\sigma_n(X^n_s)\|_{\mathcal L_2}^2 ds\\
&\quad
+2p(p-1)\int_0^t
\|X^n_s\|^{2p-4}
\|\sigma_n(X^n_s)^*X^n_s\|^2 ds
+M_t,
\end{aligned}
\]
where
\[
M_t
:=2p\int_0^t
\|X^n_s\|^{2p-2}
\langle X^n_s,\sigma_n(X^n_s)\,dW_s\rangle.
\]

By using
\[
\|\sigma^*(y)y\|\le \|\sigma(y)\|_{\mathcal L_2}\|y\|,
\]
we get
\[
\|X^n\|^{2p-4}\|\sigma^*(X^n)X^n_t\|^2
\le \|X^n\|^{2p-2}\|\sigma(X^n_t)\|_{\mathcal L_2}^2.
\]
Hence, thanks to \eqref{wp2}, we have
\[\|X^n_s\|^{2p-2}\langle X^n_s,b_n(X^n_s)\rangle+\|X^n_s\|^{2p-4}
\|\sigma_n(X^n_s)^*X^n_s\|^2\leq c_p\,\left(1+\|X^n_s\|^{2p}\right),\]
so that
\[
\|X^n_t\|^{2p}
\le \|x\|^{2p}
+c_p\int_0^t (1+\|X^n_s\|^{2p})ds
+M_t.
\]

Now, if we  take the supremum with respect to $t$ first, and   then the expectation, we obtain
\[
\mathbb E\, \sup_{t\le T}\|X^n_t\|^{2p}
\le \|x\|^{2p}
+c_p\int_0^T
\Big(1+\mathbb E\sup_{r\le s}\|X^n_r\|^{2p}\Big)ds
+\mathbb E\,\sup_{t\le T}|M_t|.
\]
Since
\[
\langle M\rangle_T
=4p^2\int_0^T
\|X^n_s\|^{4p-4}
\|\sigma_n(X^n_s)^*X^n_s\|^2 ds
\le c_p\int_0^T
\|X^n_s\|^{4p-2}
\|\sigma_n(X^n_s)\|_{\mathcal L_2}^2 ds,
\]
we obtain
\[
\mathbb E\sup_{t\le T}|M_t|
\le \frac 12\, \mathbb E\sup_{t\le T}\|X^n_t\|^{2p}
+c
\int_0^T
\left(1+\mathbb E\sup_{r\le s}\|X^n_r\|^{2p}\right)ds.
\]
This allows to conclude
\[
\mathbb E\sup_{t\le T}\|X^n_t\|^{2p}
\le c_p
\Big(
\|x\|^{2p}
+1
+\int_0^T
\mathbb E\sup_{r\le s}\|X^n_r\|^{2p} ds
\Big).
\]
and Gronwall's lemma implies
\begin{equation}\label{wp6}
\mathbb E\sup_{t\le T}\|X_t^n\|^{2p}
\le c_p(T)(1+\|x\|^{2p}),
\end{equation}
with $c_p(T)$ independent of $n$.

Next, if for every $n \in\,\mathbb{N}$ we define
\[
\tau_n=\inf\{t:\|X_t^n\|\ge n\},
\]
according to \eqref{wp6}, we have
\begin{equation}\label{wp3}
\mathbb{P}\left(\tau_\infty:=\lim_{n\to \infty}\tau_n=+\infty\right)=1.
\end{equation}
Actually, for every $T>0$ we have
\[\mathbb{P}\left(\tau_n\leq T\right)\leq \frac{1}{n^{2p}}\,\mathbb E\,\sup_{t\le T}\|X_t^n\|^{2p}\leq \frac{c_p(T)}{n^{2p}}(1+\|x\|^{2p}),\]
so that
\[\lim_{n\to\infty}\mathbb{P}\left(\tau_n\leq T\right),\]
and \eqref{wp3} follows.

In view of \eqref{wp3}, for every $(t,\omega) \in\,[0+\infty)\times \{\tau_\infty=+\infty\}$, there exists some $n=n(\omega,t) \in\,\mathbb{N}$ such that $t\leq \tau_n(\omega)$ and we  define
 \[X_t(\omega)=X^n_t(\omega).\]
 Since for every $n\leq m$ we have
 \[\|x\|\leq n\Longrightarrow b_n(x)=b_m(x)=b(x),\]
 this is a good definition and $X_t$ is a global solution of  equation \eqref{wp4}. Moreover, due to \eqref{wp6}, estimate \eqref{wp5} holds.

Finally, as far as the uniqueness is concerned,
if $X$ and $\widetilde X$ are two solutions with the same initial condition, we
define
\[
\rho_n:=\inf\{t\ge0:\ \|X_t\|\vee\|\widetilde X_t\|\ge n\}.
\]
On $[0,\rho_n]$, both processes remain in the ball of radius $n$, where $b$ and
$\sigma$ are Lipschitz; standard estimates yield $X_t=\widetilde X_t$ for all
$t\le \rho_n$ a.s. Letting $n\to\infty$ yields $X=\widetilde X$ a.s. on every
finite interval, hence globally.

\section{Explicit formula for $\bar{A_i}$}\label{form_Ai}
We recall the constants from \Cref{cor_gn3bound}:
\begin{align*}
    C_1=1+2\eta\|Db\|_\infty+\eta^2J_1^2,\quad C_i=\eta^2J_i^2,\quad i=2,3.
\end{align*}
One can estimate the denominators appearing in $A_2$ and $A_3$:
\begin{align*}
    C_1-C_1^{1/2}&=\frac{C_1-1}{1+C_1^{-1/2}}\ge\frac{C_1-1}{2}\ge\eta\|Db\|_\infty,\\
    C_4-C_1&\ge2\eta\|Db\|_\infty,\quad C_4-C_1^{1/2}=C_4-C_1\ge2\eta\|Db\|_\infty.
\end{align*}
These bounds gives the lower bound of the absolute value of the denominators. Using the above estimates, we obtain
\begin{align*}
    |A_2|&=\left(\|D^2g_0\|_\infty+\frac{\|Dg_0\|_\infty C_2^{1/2}}{C_1-C_1^{1/2}}\right)\frac{(C_1C_2)^{1/2}}{C_4-C_1},\\
    &\le\left(\|D^2g_0\|_\infty+\frac{\|Dg_0\|_\infty J_2}{\|Db\|_\infty}\right)\frac{J_2C_1^{1/2}(\eta)}{2\|Db\|_\infty}=:\bar{A_2}(\eta),\\
    |A_3|&=\frac{1}{C_4-C_1^{1/2}}\left(\frac{\|Dg_0\|_\infty C_2^{1/2}}{C_1-C_1^{1/2}}(C_1C_2)^{1/2}+\|Dg_0\|_\infty C_3^{1/2}\right)\\
    &\le\frac{1}{2\|Db\|_\infty}\left(\frac{\|Dg_0\|_\infty J_2^2 C_1^{1/2}(\eta)}{2\|Db\|_\infty}+\|Dg_0\|_\infty J_3\right)=:\bar{A_3}(\eta),
\end{align*}
and
\begin{align*}
    |A_1|&\le\|D^3g_0\|_\infty+\bar{A_2}(\eta)+\bar{A_3}(\eta)=:\bar{A_1}(\eta).
\end{align*}

\end{document}